\documentclass[11pt]{article}
\usepackage[T1]{fontenc}
\usepackage{lmodern}
\usepackage[a4paper]{geometry}
\usepackage[english,activeacute]{babel}
\usepackage{caption}
\usepackage{subcaption}
\usepackage{psfrag}
\usepackage[percent]{overpic}
\usepackage{graphicx,amsmath,amssymb,amsthm,xcolor,wasysym,bbm,enumitem}
\usepackage{upgreek}
\usepackage[active]{srcltx}
\usepackage{upgreek}
\usepackage{bbm}
\usepackage{bm}

\newcommand{\A}{\bm{A}}

\newcommand{\m}{\bm{m}}

\newcommand{\n}{\bm{n}}
\renewcommand{\b}{\bm{b}}
\newcommand{\f}{\bm{f}}
\newcommand{\mm}{\bm{\mathrm{m}}}

\newcommand{\mmm}{\mathrm{m}}

\newcommand{\vp}{\bm\varphi}
\renewcommand{\u}{\bm u}

\newcommand{\boA}{\mathcal{A}}

\newcommand{\boC}{\mathcal{C}}

\newcommand{\boO}{\mathcal{O}}
\newcommand{\boP}{\mathcal{P}}

\newcommand{\boR}{\mathcal{R}}

\newcommand{\gb}{{b}}

\newcommand{\gm}{{m}}
\newcommand{\gn}{{n}}

\newcommand{\gw}{{w}}

\newcommand{\dist}{\operatorname{dist}}

\renewcommand{\S}{\mathbb{S}}

\newcommand{\R}{\mathbb{R}}

\newcommand{\ptl}{{\partial}}

\providecommand{\abs}[1]{|#1 |}

\newcommand{\loc}{\mathrm{loc}}
\providecommand{\norm}[1]{\|#1 \|}

\renewcommand{\Re}{\operatorname{Re}}
\renewcommand{\Im}{\operatorname{Im}}

\newcommand{\bqq}{\begin{equation*}}
\newcommand{\eqq}{\end{equation*}}
\newcommand{\bq}{\begin{equation}}
\newcommand{\eq}{\end{equation}}

\newtheorem{thm}{Theorem}[section]

\newtheorem{prop}[thm]{Proposition}

\newtheorem{lemma}[thm]{Lemma}

\newtheorem{cor}[thm]{Corollary}
\newtheorem{remark}[thm]{Remark}

\renewcommand{\S}{\mathbb{S}}

\theoremstyle{definition}
\newcommand{\grad}{\nabla}

\begin{document}
	\title{Self-similar shrinkers of the one-dimensional Landau--Lifshitz--Gilbert equation}

	\author{
		\renewcommand{\thefootnote}{\arabic{footnote}}
		Susana Guti\'errez\footnotemark[1] ~and Andr\'e de Laire\footnotemark[2]}
	\footnotetext[1]{School of Mathematics,
		University of Birmingham, Edgbaston, Birmingham, B15 2TT, United
		Kingdom. E-mail: {\tt s.gutierrez@bham.ac.uk}}
	\footnotetext[2]{ 
		Univ.\ Lille, CNRS, UMR 8524, Inria - Laboratoire Paul Painlev\'e, F-59000 Lille, France.\\
		E-mail: {\tt andre.de-laire@univ-lille.fr}}
	\date{}
	\maketitle
	
	\begin{abstract}

		The main purpose of this paper is the analytical study of self-shrinker  solutions of the one-dimensional Landau--Lifshitz--Gilbert equation (LLG), a model describing the dynamics for the spin in ferromagnetic materials.  We show that there is a unique smooth family of backward self-similar solutions to the LLG equation, up to symmetries, 
		and we establish their asymptotics. Moreover, we obtain that in the presence of damping, the trajectories of the self-similar profiles converge  to great circles  on the sphere $\S^2$, at an exponential rate.
		
		In particular, the results presented in this paper provide  examples of blow-up in finite time,
		where the singularity develops due to rapid oscillations forming  limit  circles.
		
		\bigskip
		\bigskip
		
		

		82D40;35C06;35B44; 35C20;53C44;35Q55;58E20;35K55
		
	\end{abstract}

	\section{Introduction}
	\setcounter{equation}{0}
	\numberwithin{equation}{section}
	
	\subsection{The Landau--Lifshitz--Gilbert equation: self-similar solutions}
	In this paper we continue the investigation started in \cite{gutierrez-delaire1,gutierrez-delaire2} concerning the existence and properties of self-similar solutions for the Landau--Lifshitz--Gilbert equation (LLG). 
	This equation describes the dynamics for the magnetization or spin in
	ferromagnetic materials \cite{landaulifshitz,gilbert} and is given by the system of nonlinear equations
	\begin{equation}\label{LLG}
	\ptl_t \mm= \beta \mm\times \Delta \mm -\alpha \mm \times (\mm\times \Delta\mm),
	\tag{LLG}
	\end{equation}
	where  $\mm=(\mmm_1, \mmm_2, \mmm_3):\R^N \times I\longrightarrow \S^2$
	is the spin vector, $I\subset \R$, $\beta\geq 0$, $\alpha\geq 0$, $\times$ denotes
	the usual cross-product in $\R^3$, and $\mathbb{S}^2$ is the unit
	sphere in $\mathbb{R}^3$. This model for ferromagnetic materials
	constitutes a fundamental equation in the magnetic recording industry \cite{wei2012}.
	The parameters $\beta\geq 0$ and $\alpha\geq 0$ are, respectively, the so-called exchange constant and Gilbert damping,
	and take into account the exchange of energy in the system and the effect of damping on the spin chain. By considering a time-scaling, one can assume without loss of generality that the parameters $\alpha$ and $\beta$ satisfy 
	$$
	\alpha\in[0,1]\quad \text{and}\quad  \beta=\sqrt{1-\alpha^2}.
	$$
	From a purely mathematical point of view, the LLG equation is extremely interesting since it interpolates between two fundamental geometric evolution equations, the Schr\"odinger map equation and the heat flow for harmonic maps, via specific choices of the  parameters involved. Precisely, we recall that in the limit case $\alpha= 1$ (and, consequently, $\beta=0$), \eqref{LLG} reduces to the heat flow for harmonic maps onto $\mathbb{S}^2$,
	\begin{equation}\label{HFHM}\tag{HFHM}
	\ptl_t \mm-\Delta\mm=\abs{\grad{\mm}}^2\mm,
	\end{equation}
	and,  if $\alpha=0$ (no damping), it reduces to the Schr\"odinger map equation 
	\begin{equation}\label{SM}\tag{SM}
	\ptl_t \mm= \mm\times \Delta \mm.
	\end{equation}
	When $0<\alpha<1$, \eqref{LLG}  is of parabolic type. We refer the reader to  \cite{lakshmanan,guo,gutierrez-delaire1, gutierrez-delaire2,deLaire4,deLaGra1,deLaGra2,deLaGra3} and the references therein for more details and surveys on these equations.
	
	A natural question, that has proved relevant to the  understanding of the global behavior of solutions and formation of singularities, is whether or not there exist solutions which are invariant under scalings of the equation. In the case of the LLG equation it is straightforward to see that the equation is invariant under the following scaling:  If  $\mm$  is a solution of \eqref{LLG}, then   $\mm_{\lambda}(t,x)=\mm(\lambda x,\lambda^2 t)$, for any positive number $\lambda$, is also a solution. Associated with this invariance, a solution $\mm$ of \eqref{LLG} defined on $I=\R^+$ or $I=\R^-$ is called {\it{self-similar}} if it is invariant under rescaling, that is
	$$
	\mm(x,t)= \mm(\lambda x,\lambda^2 t),\quad \forall \lambda>0, \quad \forall x\in \R^N,\quad  \forall t\in I.
	$$
	Fixing $T\in \R$ and performing a translation in time, this definition leads to  two types of self-similar solutions: 
	A forward self-similar solution or {\em expander} is a solution of the form
	\begin{equation}\label{f-expander}
	\mm(x,t)=\f\left(\frac{x}{\sqrt{t-T} }\right),\qquad {\hbox{for}}\quad (x,t)\in \R^N\times(T,\infty),
	\end{equation}
	and  a backward self-similar solution or {\em shrinker}  is a solution of the form
	\begin{equation}\label{f-profile}
	\mm(x,t)=\f\left(\frac{x}{\sqrt{T-t} }\right),\qquad {\hbox{for}}\quad (x,t)\in \R^N\times(-\infty, T),
	\end{equation}
	for some profile $\f:\mathbb{R}^N\longrightarrow \mathbb{S}^2$. In this manner, expanders evolve from a singular value at time $T$, while shrinkers evolve towards a singular value at time $T$.

	Self-similar solutions have received a lot of attention in the study of nonlinear PDEs
	because they can provide important information about the dynamics of the equations.
	While expanders are related to non-uniqueness phenomena, resolution of singularities and 
	long time description of solutions, shrinkers 
	are often related to phenomena of singularity formation (see e.g.\ \cite{book-giga,eggers}). On the other hand, the construction and understanding of the dynamics and properties of self-similar solutions also provide an idea of which are the natural spaces to develop a well-posedness theory that captures these often very physically relevant structures.  
	Examples of equations for which self-similar solutions have  been studied include, among others, the Navier--Stokes equation, semilinear parabolic equations, and geometric flows such as Yang--Mills, mean curvature flow and harmonic map flow. We refer to \cite{jia-tsai,quittner,ilmanen,struwe,bidaut-veron} and the references therein for more details.
	
	
	Although the results that will be presented in this paper relate to self-similar shrinkers of the one-dimensional LLG equation (that is, to solutions $\mm:\mathbb{R}\times I\longrightarrow \S^2$ of  LLG), for the sake of context we describe some of the most relevant results concerning maps from $\R^N\times I$ into $\S^d$, with $N\geq 2$ and $d\geq 2$.
	In this setting one should point out that the majority of the works in the literature concerning  the study of self-similar solutions of the  LLG equation are confined to the heat flow for harmonic maps equation, i.e.\ $\alpha=1$. In the case when $\alpha=1$,  the main works on the subject restrict the analysis to corotational maps
	taking values in $\S^d$, which reduces the analysis of \eqref{HFHM} to the study of a second order real-valued ODE. Then tools such as the maximum principle or the shooting method can be used to show  the existence of solutions. We refer to \cite{fan,gastel,germain-rupflin,biernat-donninger,bizon-wasserman,biernat-bizon,germain-ghoul}
	and the references therein for more details on such  results for 
	maps taking values in $\S^d$, with $d\geq 3.$ Recently, 
	Deruelle and Lamm~\cite{deruelle-lamm} have studied the Cauchy problem 
	for  the harmonic map heat flow with initial data $\mm^0 : \R^N\to \S^d$, with $N\geq 3$
	and $d\geq 2$, where $\mm^0$ is a Lipschitz 0-homogeneous function, homotopic to a constant,
	which implies the existence of expanders coming out of $\mm^0$. 
	
	When $0<\alpha\leq 1$, the existence of self-similar {expanders} for the LLG equation was recently established by the authors in \cite{gutierrez-delaire2}. This result  is a consequence of a well-posedness theorem for the  LLG equation 
	considering an initial data $\mm^0 : \R^N\to \S^2$ in the space BMO
	of functions of bounded mean oscillation.
	Notice that this result  includes in particular the case of the  harmonic map heat flow. 
	
	As mentioned before, in the absence of damping ($\alpha=0$), \eqref{LLG} reduces to  the Schr\"odinger map equation \eqref{SM}, which is reversible in time, so that the notions of expanders and shrinkers coincide.
	For this equation, Germain, Shatah and Zeng \cite{germain-shatah-zeng} established the
	existence of ($k$-equivariant) self-similar profiles $\f : \R^2\to \S^2$.
	
	\subsection{Goals and statements of main results}
	The results of this paper aim to advance our understanding of self-similar solutions
	of the {\em one-dimensional} LLG equation.  
	In order to contextualize and motivate our results, we continue to provide further details 
	of what is known about  self-similar solutions 
	in  this context.
	
	In the 1d-case, when $\alpha=0$, \eqref{SM} is closely related to the Localized Induction Approximation (LIA), and self-similar profiles $\f: \R\to \S^2$ 
	were obtained and analyzed in \cite{vega-gutierrez,vega-gutierrez1,lakshmanan0,buttke}.
	In the context of LIA, self-similar solutions constitute a uniparametric family of smooth solutions that develop a singularity in the shape of a corner in finite time. For further work related to these solutions, including the study of their continuation after the blow-up time and their stability, we refer to the reader to \cite{banica-vega-3,banica-vega}. At the level of the Schr\"odinger map equation, these self-similar solutions provide examples of smooth solutions that develop a jump singularity in finite time.  
	
	In the general case  $\alpha\in [0,1]$,  the analytical study of self-similar expanders of the one-dimensional \eqref{LLG} was carried out in \cite{gutierrez-delaire1}. Here, it was shown that these solutions  are given by  a family of  smooth 
	profiles  $\{  \f_{c,\alpha}\}_{c,\alpha}$, and that the corresponding  expanders are  associated with a discontinuous (jump) singular initial data.
	We refer to \cite{gutierrez-delaire1,gutierrez-delaire2} for the precise statement of this result, and  the stability of these solutions, as well as the qualitative and quantitative analysis of their dynamics with respect to the parameters $c$ and $\alpha$.
	
	It is important to notice that in the presence of damping ($\alpha>0$), since 
	the LLG equation is not time-reversible, the notion of expander is different from that of shrinker.
	It is therefore natural to ask the following question:  {\em{What can be said about shrinker solutions for the one-dimensional LLG equation?}}
	
	Answering this question constitutes the main purpose of this paper.  Precisely, our main goals  are to establish the classification of  self-similar shrinkers of the one-dimensional LLG equation of the form \eqref{f-profile}
	for some profile $\f:\R\to \mathbb{S}^2$, and the analytical study  of their properties. 
	In particular, we will be especially interested in studying the dynamics of these solutions as $t$ tends to the time of singularity $T$,  and understanding  how the dynamical behavior of these solutions is affected by the presence of damping. Since, as it has been already mentioned,  the case  $\alpha=0$ has been previously considered in the literature (see \cite{banica-vega-3,gutierrez-survey}), in what follows we will assume that $\alpha\in(0,1]$.
	
	In order to state our first result, we observe that if  $\m$ 
	is a  solution to \eqref{LLG} of the form \eqref{f-profile} for some smooth 
	profile $\f$, then  
	$\f$ solves the  following system of ODEs
	\begin{equation}
	\label{eq:self-similar-intro}
	\frac{x \f'}2=\beta \f\times \f''-\alpha \f\times (\f\times \f''),
	\quad\textup{ on }\R,
	\end{equation}
	which   recasts as
	\begin{equation}
	\label{eq:EDO:intro}
	\alpha \f''+\alpha\abs{\f'}^2\f +\beta (\f\times \f')'-\frac{x \f'}2=0,
	\quad\textup{ on }\R,
	\end{equation}
	due to the fact that $\f$ takes values in $\S^2$.
	
	In the case $\alpha\in (0,1)$, it seems unlikely to be able to find explicit solutions to 
	\eqref{eq:EDO:intro}, and even their existence is not clear  (see also equation \eqref{eq-f-j}). Nevertheless, surprisingly we can establish the following rigidity result concerning the possible weak solutions to \eqref{eq:EDO:intro} (see Section~\ref{sec:EDO} for the definition of weak solution).
	\begin{thm}
		\label{thm:EDO}
		Let $\alpha\in (0,1]$. Assume that $\f$ is a weak solution to \eqref{eq:EDO:intro}. 
		Then $\f$ belongs to $\boC^\infty(\R;\S^2)$ and there exists $c\geq 0$ such that
		$\abs{\f'(x)}=ce^{\alpha x^2/4}$, for all $x\in \R$. 
	\end{thm}
	Theorem~\ref{thm:EDO} provides a necessary condition on the possible (weak) solutions of \eqref{eq:EDO:intro}: namely the modulus of the  gradient of any  solution {\em must} be $ce^{\alpha x^2/4}$, for some $c\geq 0$. We proceed now to establish the {\em existence} of solutions satisfying this condition
	for any $c>0$ (notice that the case when $c=0$ is trivial).
	
	To this end,  we will follow a  geometric
	approach that was proven to be very fruitful in similar contexts 
	(see e.g.~\cite{lakshmanan0,lipniacki,lamb,vega-gutierrez,demontis-ortenzi}),
	including the work of the authors in the study of expanders  \cite{gutierrez-delaire1}. 
	As explained in Subsection~\ref{subsec:existence},  this approach relies on identifying  $\f$
	as the unit tangent vector $\m:=\f$ of a curve  $\bm X_{\m}$ in $\mathbb{R}^3$ parametrized by arclength. 
	Thus, assuming that $\f$ is a solution to \eqref{eq:EDO:intro}
	and using the Serret--Frenet system associated with the curve $\bm X_{\m}$, 
	we can deduce  that the curvature and the torsion are explicitly  given by \bq
	\label{curvature}
	k(x)=c e^{{\alpha x^2}/{4}}, \quad  \text{ and } \quad 
	\tau(x)=-\frac{\beta x}{2},\eq
	respectively, for some $c\geq 0$ (see~Subsection~\ref{subsec:existence} for further details).
	In particular, we have  $\abs{\m'(x)}=k(x)=c e^{{\alpha x^2}/{4}}$,
	in agreement with Theorem~\ref{thm:EDO}. 
	Conversely, given  $c\geq 0$ and  denoting $\m_{c,\alpha}$ the solution of
	the Serret--Frenet system 
	\begin{equation}\label{serret}
	\left\{
	\begin{aligned}
	\m'(x)&=k(x)\n(x),\\
	\n'(x)&=-k(x)\m(x) +\tau(x) \b(x),\\
	\b'(x)&=-\tau(x) \n(x),
	\end{aligned}
	\right.
	\end{equation}
	with curvature  and torsion as in \eqref{curvature}, 
	and   initial conditions (w.l.o.g.) 
	\begin{equation}\label{IC} {\m}(0)=(1,0,0), \quad
	{\n}(0)=(0,1,0), \quad {\b}(0)=(0,0,1),
	\end{equation}
	we obtain a solution to \eqref{eq:EDO:intro}. Moreover, we can show that the solutions constructed in this manner provide, up to symmetries, all the solutions
	to \eqref{eq:self-similar-intro}. The precise statement is the following.
	\begin{prop}
		\label{prop:unicidad}
		The set of nonconstant solutions to 
		\eqref{eq:self-similar-intro} is $\{\boR \m_{c,\alpha} : c>0,\boR \in SO(3)\}$,
		where $SO(3)$ is the group  of rotations about the origin preserving orientation.
	\end{prop}
	
	The above proposition reduces the study of self-similar shrinkers to the understanding  
	of the family of self-similar shrinkers associated with the 
	profiles $\{\m_{c,\alpha}\}_{c,\alpha}$. 
	The next result summarizes 
	the properties of these solutions.
	\begin{thm}
		\label{thm-conv}
		Let $\alpha\in(0,1]$, $c>0$, $T\in \R$
		and ${\m}_{c,\alpha}$ be the solution of the Serret--Frenet
		system \eqref{serret}  with initial conditions \eqref{IC},
		$$
		k(x)=c e^{{\alpha x^2}/{4}}\qquad \text{and} \qquad \tau(x)=
		-\dfrac{\sqrt{\vphantom{\alpha^A} {1-\alpha^2}} x}{2}.
		$$
		Define
		\begin{equation} \label{def-m}
		\mm_{c, \alpha}(x,t) =\m_{c,\alpha}\left( \frac{x}{\sqrt{T-t}}  \right),   \qquad t<T.
		\end{equation}
		Then we have the following statements.
		\begin{enumerate}[label=({\roman*}),ref={\it{({\roman*})}}]
			\item\label{regular} The function $\mm_{c, \alpha}$ belongs to  $\mathcal{C}^\infty( \R\times (-\infty,T);\mathbb{S}^2)$,  solves  \eqref{LLG} for $t\in (-\infty,T)$,
			and
			$$\abs{\partial_{x} \mm_{c, \alpha}(x,t)}=\frac{c}{\sqrt{T-t}}e^{\frac{\alpha x^2}{4(T-t)}},$$
			for all $(x,t)\in \R\times(-\infty,T)$.
			
			\item\label{parity}
			The components of the profile  $\m_{c,\alpha}=(m_{1,c,\alpha},m_{2,c,\alpha},m_{3,c,\alpha})$ satisfy that $m_{1,c,\alpha}$ is even,  while
			$m_{2,c,\alpha}$ and $m_{3,c,\alpha}$ are odd.
			\item\label{asymptotics}
			There exist constants $\rho_{j,c,\alpha}\in[0,1],$
			$B_{j,c,\alpha}\in [-1,1],$
			and 
			$\phi_{j,c,\alpha} \in [0,2\pi),$
			for $j\in \{1,2,3\}$,
			such that
			we have the  following asymptotics for the profile   $\m_{c,\alpha}=(m_{1,c,\alpha},m_{2,c,\alpha},m_{3,c,\alpha})$ and its derivative:
			\begin{equation}
			\label{eq:asymp}
			\begin{aligned}
			m_{j,c,\alpha}(x)=&\rho_{j,c,\alpha}\cos(c\Phi_\alpha(x)-\phi_{j,c,\alpha})
			-\frac{\beta B_{j,c,\alpha}}{2c} xe^{-\alpha x^2/4}\\
			&+
			\frac{\beta^2 \rho_{j,c,\alpha}}{8c}\sin(c\Phi_\alpha(x)-\phi_{j,c,\alpha})
			\int_x^\infty s^2 e^{-\alpha s^2/4}ds
			+\frac{\beta}{\alpha^5c^2}\boO(x^2e^{-\alpha x^2/2}),
			\end{aligned}
			\end{equation}
			and
			\begin{equation}
			\label{eq:asymp:der}
			\begin{aligned}
			m'_{j,c,\alpha}&(x)=- c\rho_{j,c,\alpha}  \sin(c\Phi_\alpha(x)-\phi_{j,c,\alpha}) e^{\alpha x^2/4}
			\\
			&+
			\frac{\beta^2 \rho_{j,c,\alpha}}{8}\cos(c\Phi_\alpha(x)-\phi_{j,c,\alpha})
			e^{\alpha x^2/4}\int_x^\infty s^2 e^{-\alpha s^2/4}ds
			+\frac{\beta}{\alpha^5c}\boO(x^2e^{-\alpha x^2/4}),
			\end{aligned}
			\end{equation}
			for all $x\geq 1$, where 
			$$
			\Phi_\alpha(x)=\int_{0}^{x} e^{\frac{\alpha s^2}{4}}\, ds.
			$$
			Moreover, the constants   $\rho_{j,c,\alpha}$, $B_{j,c,\alpha}$, and  $\phi_{j,c,\alpha}$ satisfy the following identities
			$$
			\rho_{1,c,\alpha}^2+\rho_{2,c,\alpha}^2+\rho_{3,c,\alpha}^2=2,\quad 
			B_{1,c,\alpha}^2+B_{2,c,\alpha}^2+B_{3,c,\alpha}^2=1 \quad \text{and} \quad \rho_{j,c,\alpha}^2+B_{j,c,\alpha}^2=1,\quad  j\in \{1,2,3\}.
			$$
			
			\item\label{convergence}
			The solution $\mm_{c,\alpha}=(\mmm_{1,c,\alpha},\mmm_{2,c,\alpha},\mmm_{3,c,\alpha})$ satisfies the following pointwise convergences
			\begin{equation}
			\label{convergencia}
			\begin{aligned}
			\lim_{t\to T^-}(\mmm_{j,c,\alpha}(x,t)-\rho_{j,c,\alpha}\cos\big(  c\, \Phi_{\alpha} \big( \frac{x}{\sqrt{T-t}}\big) -\phi_{j,c,\alpha}  \big)=0, \text{ if } x>0,\\
			\lim_{t\to T^-}(\mmm_{j,c,\alpha}(x,t)-\rho_{j,c,\alpha}^{-}\cos\big(  c\, \Phi_{\alpha} \big( \frac{-x}{\sqrt{T-t}}\big) -\phi_{j,c,\alpha}  \big)=0, \text{ if }x<0,
			\end{aligned}
			\end{equation}
			for $j\in \{1,2,3\}$, where $\rho_{1,c,\alpha}^-=\rho_{1,c,\alpha}$, 
			$\rho_{2,c,\alpha}^-=-\rho_{2,c,\alpha}$ and $\rho_{3,c,\alpha}^-=-\rho_{3,c,\alpha}$.
			\item\label{IVP} For any  $\vp\in W^{1,\infty}(\R;\R^3)$, we have
			\begin{equation*}
			\lim_{t\to T^-}\int_\R \mm_{c,\alpha}(x,t)\cdot \vp(x)dx=0.
			\end{equation*}
			In particular, $\mm_{c,\alpha}(\cdot, t)\to 0$ as $t\to T^-$, as a tempered  distribution.
		\end{enumerate}
	\end{thm}
	
	\medskip
	
	\noindent It is important to remark that Theorem~\ref{thm-conv} provides examples of (smooth) solutions to the 1d-LLG equation that blow up in finite time. In order to see this, 
	let us first recall that the existence of smooth solutions to \eqref{LLG} on short times 
	can be established as in the case of the heat flow for harmonic maps \cite{lin}, 
	using that \eqref{LLG} is a strongly parabolic system \cite{guo-hong,amann}. In particular, in the one-dimensional case, for any initial condition $\mm^0 \in \boC^\infty(\R,\S^2)$, there exists a maximal time 
	$0< T_{\max}\leq\infty$ such that \eqref{LLG} admits a unique, smooth solution 
	$\mm \in \boC^\infty(\R\times [0,T_{\max});\S^2)$. Moreover, if $T_{\max}<\infty$, then 
	$$\lim_{t\to T_{\max}^-}\norm{\partial_{x}\mm(\cdot,t)}_{L^\infty(\R)}=\infty.$$
	Next, observe that for any $c>0$ and $T\in \R$, the solution of the initial value problem associated with  \eqref{LLG} and  with  initial condition 
	$\m_{c,\alpha}(\cdot)$ at time $T-1$ is given  by
	$\mm_{c,\alpha}$ in Theorem~\ref{thm-conv}, for $t\in [T-1,T)$, and  blows up  at time $T$.
	Indeed, from \ref{regular} in Theorem~~\ref{thm-conv} , we have that 
	$$ \lim_{t\to T^-}\abs{\partial_{x} \mm_{c, \alpha}(x,t)}=
	\lim_{t\to T^-}
	\frac{c}{\sqrt{T-t}}e^{\frac{\alpha x^2}{4(T-t)}}=\infty,
	$$
	for $c>0$ and for all $x\in \R$.
	
	Notice also that from the asymptotics in part \ref{asymptotics}  and the symmetries of the profile established in part \ref{parity}, we obtain a precise description of the fast oscillating nature of the blow up of the solution \eqref{def-m} given in Theorem~\ref{thm-conv}. In this setting, we observe  that part \ref{asymptotics} of the above theorem provides the asymptotics of the profile 
	$\m_{c,\alpha}$ at infinity, in terms of a  fast oscillating principal part, plus some exponentially decaying terms. Notice that for the integral term in \eqref{eq:asymp},
	we have  (see e.g.\ \cite{abram})
	\bqq
	\int_x^\infty s^2 e^{-\alpha s^2/4}ds
	=\frac{2 x e^{-\alpha x^2/4} }{\alpha}\Big(
	1+\frac{2}{\alpha x^2}-\frac4{\alpha^2x^4}+\cdots
	\Big),
	\quad  \text{ as }x\to \infty,
	\eqq 
	and that using the asymptotics for the Dawson's integral \cite{abram}, 
	we also get 
	$$\Phi_{\alpha}(x)=\frac{2  e^{\alpha x^2/4} }{\alpha x}\Big(
	1+\frac{2}{\alpha x^2}+\frac{12}{\alpha^2x^4}+\cdots
	\Big),
	\quad  \text{ as }x\to \infty.
	$$ 
	It is also important to mention that  the big-$\boO$ in the asymptotics \eqref{eq:asymp} does not depend on the parameters, i.e.~there exists a universal constant $C>0$, such that 
	the big-$\boO$ in \eqref{eq:asymp} satisfies
	$$\abs{\boO(x^2e^{-\alpha x^2/2})}\leq C  x^2e^{-\alpha x^2/2},\quad \text{ for all }x\geq 1.$$ 
	In this manner, the constants multiplying the big-$\boO$ are meaningful and in particular,  big-$\boO$  
	vanishes when $\beta=0$ (i.e.\ $\alpha=1$).  
	
	In Figure~\ref{fig-curvas} we have depicted the profile $\m_{c, \alpha}$ 
	for $\alpha=0.5$ and $c=0.5$, where we can see their oscillating behavior. Moreover, the plots in Figure~\ref{fig-curvas} suggest  that the 
	limit sets of the trajectories are great circles on the sphere $\S^2$ when $x\to\pm\infty$. 
	This is indeed the case.  In our last result  we establish  analytically that $\m_{c,\alpha}$ 
	oscillates in a plane passing through the origin whose normal vector is given by 
	$\bm B_{c,\alpha}^{+}=(B_{1,c,\alpha},B_{2,c,\alpha},B_{3,c,\alpha})$, and $\bm B_{c,\alpha}^{-}=(-B_{1,c,\alpha},B_{2,c,\alpha},B_{3,c,\alpha})$  as  $x\to +\infty$ and  $x\to -\infty$, respectively.
	\begin{figure}[h]
		\begin{subfigure}[b]{0.33\textwidth}
			\begin{overpic}[scale=0.5]{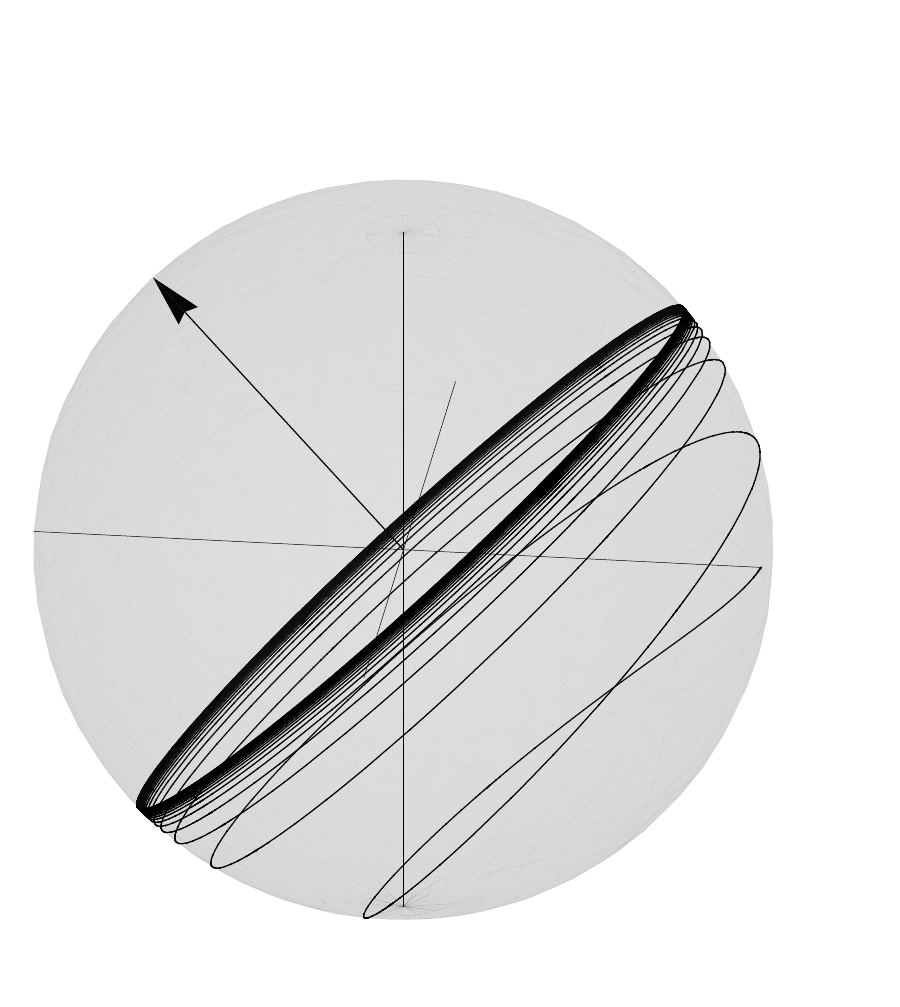}
				\put(79,44){\footnotesize{$m_1$}}
				\put(48,64){\footnotesize{$m_2$}}
				\put(40,85){\footnotesize{$m_3$}}
				\put(10,74.5){\footnotesize{$\bm B_{c,\alpha}^+$}}
			\end{overpic}
		\end{subfigure}%
		\begin{subfigure}[b]{0.33\textwidth}
			\begin{overpic}[scale=0.5]{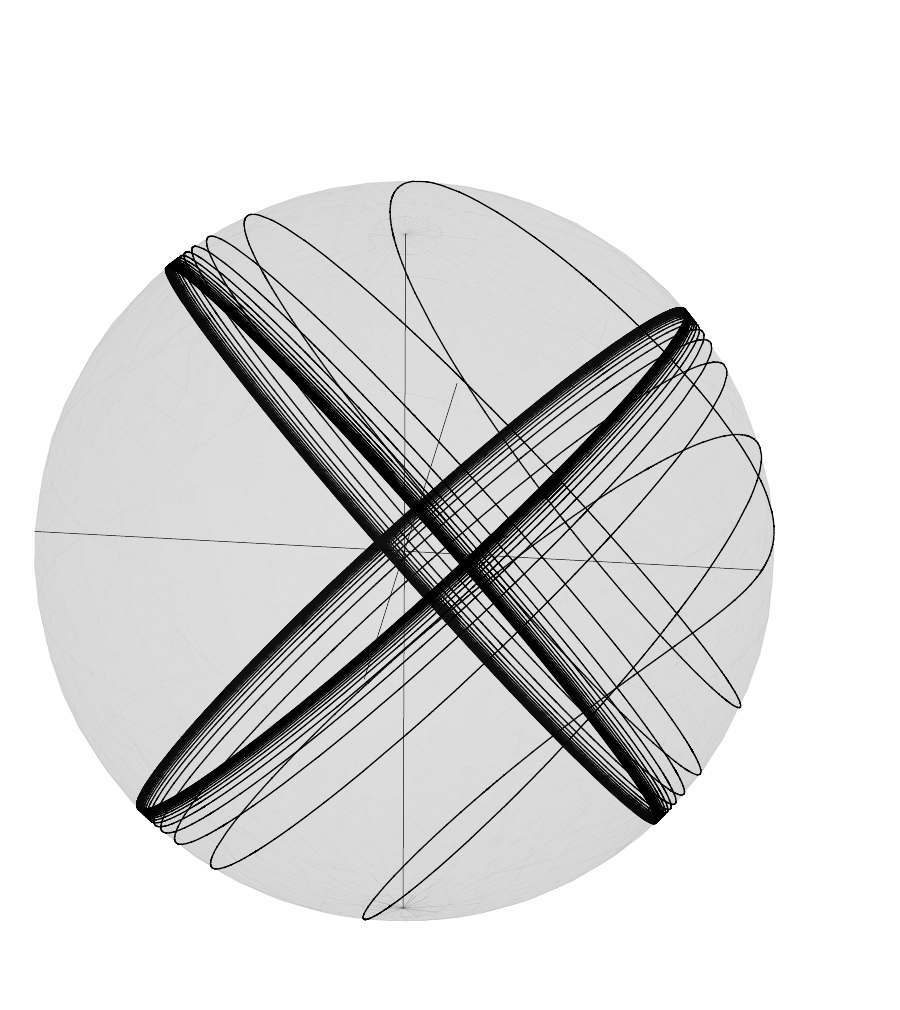}
				\put(79,44){\footnotesize{$m_1$}}
				\put(48,64){\footnotesize{$m_2$}}
				\put(40,85){\footnotesize{$m_3$}}
			\end{overpic}
		\end{subfigure}%
		\begin{subfigure}[b]{0.33\textwidth}
			\begin{overpic}[scale=0.5]{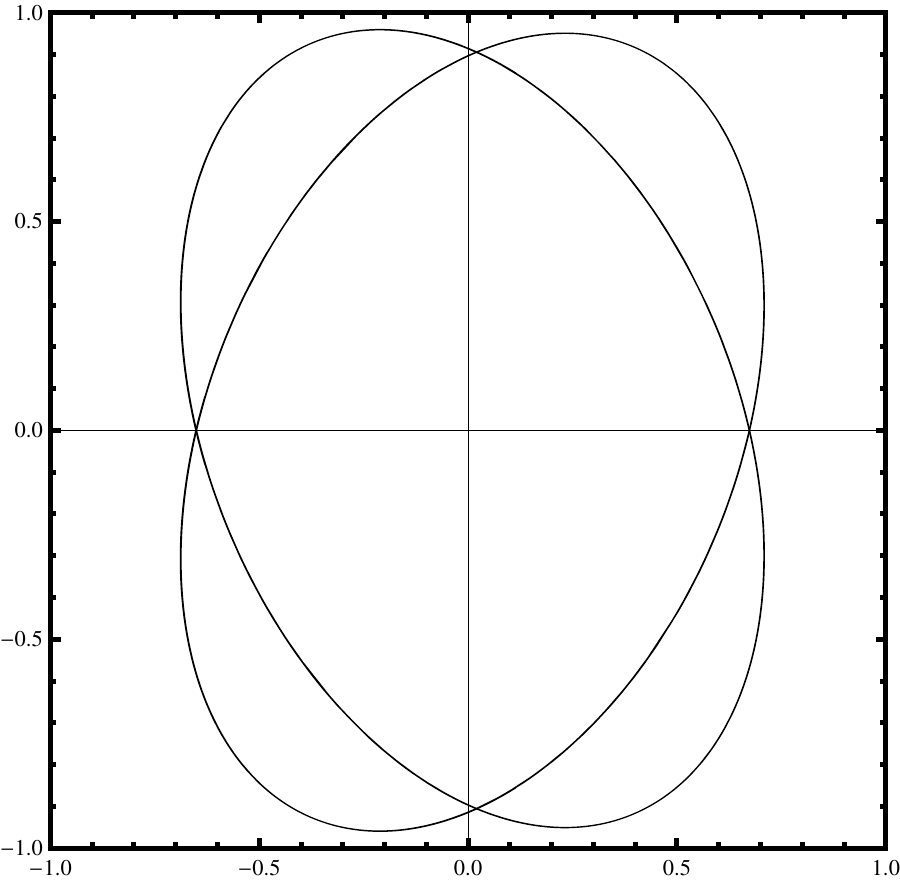}
				\put(49,-4){\footnotesize{$m_1$}}
				\put(-9,49){\footnotesize{$m_2$}}
			\end{overpic}
		\end{subfigure}%
		\caption{Profile $\m_{c,\alpha}$ for $c=0.5$ and $\alpha=0.5$. The figure on the left depicts profile for $x\in \R^+$ and the normal vector $\bm B_{c,\alpha}^{+}\approx (-0.72, -0.3, 0.63)$. The figure on the center  shows  the profile for $x\in\R$; the angle between the circles $\boC_{c,\alpha}^\pm$ is  $\vartheta_{c,\alpha}\approx 1.5951$.
			The figure on the right represents the projection of limit cycles $\boC_{c,\alpha}^\pm$ on the plane.}
		\label{fig-curvas}
	\end{figure}
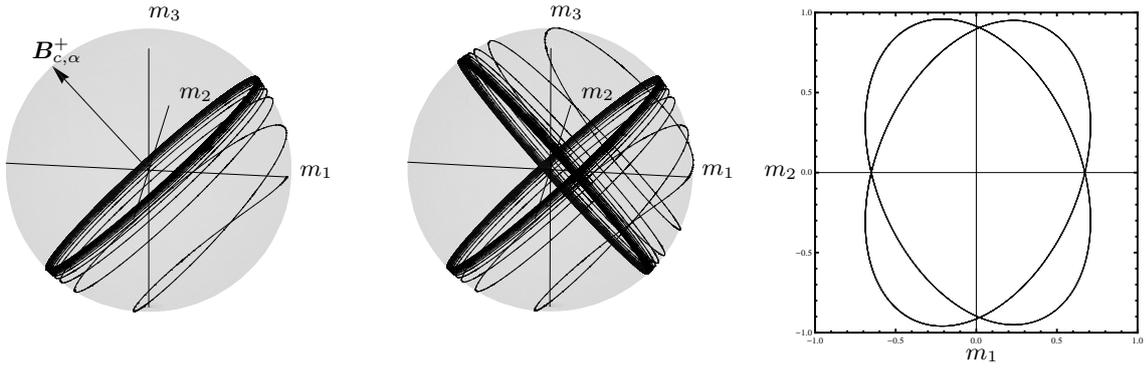
	%
	\begin{thm}
		\label{plane}
		Using the constants given in Theorem~\ref{thm-conv}, let  $\boP_{c,\alpha}^{\pm}$ be 
		the planes passing through the origin with normal vectors $\bm B_{c,\alpha}^{+}$ and  $\bm B_{c,\alpha}^{-}=(-B_{1,c,\alpha},B_{2,c,\alpha},B_{3,c,\alpha})$,
		respectively.
		Let
		$\boC_{c,\alpha}^\pm$ be the circles in $\R^3$ given by
		the intersection of these planes with the sphere, i.e.\  
		$\boC_{c,\alpha}^\pm=\boP_{c,\alpha}^{\pm}\cap \S^2.$
		Then the following statements hold.
		\begin{enumerate}[label=({\roman*}),ref={\it{({\roman*})}}]
			\item 
			For all $\abs{x}\geq 1$, we have
			\begin{equation}
			\label{dist1}
			\dist(\m_{c,\alpha}(x),\boC_{c,\alpha}^{\pm})\leq \frac{30\sqrt{2}\beta}{c \alpha^2 }\abs{x}e^{-\alpha x^2/4}.
			\end{equation}
			In particular
			\begin{equation}
			\label{dist2}
			\begin{aligned}
			\lim_{t\to T^-}\dist(\mm_{c,\alpha}(x,t),\boC_{c,\alpha}^+)=0, \quad \text{ if }x>0,\\
			\lim_{t\to T^-}\dist(\mm_{c,\alpha}(x,t),\boC_{c,\alpha}^-)=0, \quad \text{ if }x<0.
			\end{aligned}
			\end{equation}
			\item
			Let $\vartheta_{c,\alpha}=\arccos(1-2B_{1,c,\alpha}^2)$ be the angle between the circles $\boC_{c,\alpha}^\pm$.
			For  $c\geq \beta \sqrt{\pi}/\sqrt{\alpha}$, we have
			\begin{equation}
			\label{thm:est-angulo}
			\vartheta_{c,\alpha}\geq \arccos\left( -1+\frac{2\pi\beta^2}{c^2\alpha}\right).
			\end{equation}
			In particular
			\begin{equation}
			\label{thm:limit-angle-c-infty}
			\lim_{c\to\infty}  \vartheta_{c,\alpha}=\pi, \ \textup{ for all }\alpha\in (0,1],
			\quad { and }\quad 
			\lim_{\alpha \to 1}  \vartheta_{c,\alpha}=\pi, 
			\ \textup{ for all }c >0.
			\end{equation}
		\end{enumerate}
	\end{thm}
	The above theorem above  establishes the  convergence  of the limit sets of the trajectories of the profile  $\m_{c,\alpha}$ to the great circles $\boC_{c,\alpha}^\pm$ as shown in Figure~\ref{fig-curvas}. Moreover, \eqref{dist1}
	gives us an exponential rate for this convergence.
	In terms of the solution $\mm_{c,\alpha}$ to the LLG equation, Theorem~\ref{plane} provides a more precise 
	geometric information about the way that the solution blows up at time $T$,  as seen in \eqref{dist2}. The existence of limit circles for related ferromagnetic models have been 
	investigated for instance in \cite{waldner,broggi-meier} but to the best of our knowledge, 
	this is  the first time that this type  of phenomenon  has been observed  for the LLG equation.
	In Figure~\ref{fig-curvas} can see that   $\vartheta_{c,\alpha}\approx 1.5951$ for $\alpha=0.5$
	and $c=0.5$, where we have chosen the value of $c$ such that the angle is close to $\pi/2$.
	
	Finally, \eqref{thm:est-angulo}  and \eqref{thm:limit-angle-c-infty} in Theorem~\ref{plane} provide
	some geometric information 
	about behavior of the limit circles with respect to the parameters $c$ and $\alpha$.
	In particular,  formulae \eqref{thm:limit-angle-c-infty} states that the angle between the limiting circles $\boC_{c,\alpha}^+$  and $\boC_{c,\alpha}^-$  is $\pi$ as $c\to\infty$, for fixed $\alpha\in(0,1]$, and the same happens as $\alpha\to 1$, for fixed $c>0$. 
	In other words, in these two cases the circles $\boC_{c,\alpha}^\pm$ are the same (but differently oriented).

	\subsection{Comparison with the limit cases ${\bm{\alpha=0}}$ and ${\bm{\alpha=1}}$}
	It is well known that
	the Serret--Fenet system can be written as a second-order differential equation. For instance,
	if $(\m,\n,\b)=(m_j,n_j,b_j)_{j=1}^3$ is a solution of \eqref{curvature}--\eqref{serret}, using Lemma~3.1 in \cite{gutierrez-delaire1}, we have that new variable
	\bqq
	g_j(s)=e^{\frac12\int_0^sk(\sigma)\eta_j(\sigma)\,d\sigma}, \quad
	\text{ with }\quad \eta_j(x)=\frac{n_j(x)+i b_j(x)}{1+m_j(x)},
	\eqq satisfies the equation, for $j\in\{1,2,3\}$,
	\bq
	\label{eq-f-j}
	g_j''(x)-\frac{x}2 (\alpha+i\beta) g_j'(x)+ \frac{c^2}{4} e^{\alpha x^2/2}g_j(x)=0. 
	\eq
	Then, in the case  $\alpha=1$,  it easy to check (see also Remark~\ref{rem-alpha-1}) that
	the profile is explicitly given by the plane curve 
	\begin{align}
	\label{solution-alpha-1}
	\m_{c,1}(x)&=(\cos(c\Phi_1(x)),\sin(c\Phi_1(x)), 0).
	\end{align}
	In particular, we see  that the asymptotics in Theorem~\ref{thm-conv} are satisfied with
	\begin{gather*}
	\rho_{1,c,1}=1, \quad  \rho_{2,c,1}=1, \quad \rho_{3,c,1}=0,\quad 
	\phi_{1,c,1}=0, \quad \phi_{2,c,1}=3\pi/2,\quad \phi_{3,c,1}\in [0,2\pi).
	\end{gather*}
	
	The case $\alpha=0$ is more involved, but 
	using \eqref{eq-f-j}, the solution $\{\m_{c,0},\n_{c,0},\b_{c,0}\}$ of the system \eqref{serret}
	can still be  explicitly determined in terms  of confluent hypergeometric functions. 
	This  leads to the asymptotics
	\cite{vega-gutierrez,gutierrez-delaire1,gamayun-lisovyy}
	\begin{align}
	\label{solution-alpha-0}
	\m_{c,0}(x)&=\bm A_{c}-  \frac{2c}{x}\bm B_{c}\cos \left( \frac{x^2}{4}+c^2\ln(x)+\frac{\pi}2 \right)
	+O\left( \frac1{x^2} \right),
	\end{align}
	as $x\to\infty$, for some vectors $\A_{c}\in \S^2$ and $\bm B_{c}\in \R^3$.
	In particular, we see that $\m_{c,0}(x)$ converges to some vector $\A_{c}$,  as $x\to\infty$.
	Hence, there is a drastic change  in the behavior of the profile in 
	the cases $\alpha=0$ and $\alpha>0$: 
	In the first case $\m_{c,0}$ converges to a point at infinity, 
	while in the second case \eqref{dist1} tells us that $\m_{c,\alpha}$ converges to a great circle. In this sense,  there is a discontinuity in the behavior of
	$\m_{c,\alpha}$ at $\alpha=0$.

	Also, from equation  \eqref{eq-f-j}, we can formally deduce that the difference  between the expanders and shrinkers corresponds to flipping the sign in the parameters $\alpha\to -\alpha$ and $\beta \to -\beta$. Notice that the 
	exponential coefficient in  \eqref{eq-f-j}
	is	proportional  to the square of the curvature, given by $ce^{-\alpha x^2/4}$
	for the skrinkers, and $ce^{\alpha x^2/4}$ for the expanders.
	We used equation \eqref{eq-f-j} (with flipped signs)
	to obtain the asymptotics of the expanders in \cite{gutierrez-delaire1}, 
	relying on the fact the exponential term in equation  vanishes as $x\to\infty$. However,
	the exponential grow  in the case of skrinkers in \eqref{eq-f-j} changes the behavior of the solution and we cannot use the methods introduced   in \cite{gutierrez-delaire1}.

	Going back to Theorem~\ref{thm-conv}, 
	it is seems very difficult to get asymptotics for the  constants in  \eqref{eq:asymp}.
	Our strategy for the constants appearing in the  asymptotics  for the expanders in \cite{gutierrez-delaire1} relied on  obtaining  uniform estimates and using continuity arguments. In particular, using the fact that the constants in \eqref{solution-alpha-0}
	are explicit, we were able to get a  good information about the constants 
	in the asymptotics when $\alpha$ was close to $0$.
	Due to the above mentioned  discontinuity of  $\m_{c,\alpha}$ at $\alpha=0$, 
	it seems unlikely that the use of continuity arguments will provide information for the constants in the asymptotics for the shrinkers.

	Finally, let us also remark that we cannot use continuation arguments 
	to find the  behavior of the circles for $c$ small. This is expected since  $\m_{0,\alpha}(x)=(1,0,0)$ for all $x\in\R$, when $c=0$ (see \eqref{sol-c-0}).
	In Section~\ref{sec-limit-circle} we give 
	some numerical simulations
	for $c$ small.
	\bigskip
	
	\noindent {\bf{Structure of the paper. }} The outline of this paper is the following.   In Section~\ref{sec:EDO}, we study \eqref{eq:EDO:intro} as an elliptic quasilinear system and prove the rigidity result Theorem~\ref{thm:EDO}. By using the Serret--Frenet system, we prove there existence and uniqueness of solution, up to a rotation, in Section~\ref{sec:existence}. We also use this system to obtain the asymptotics 
	of the self-similar profiles. Finally, Section~\ref{sec-limit-circle} is devoted to the proof of Theorem~\ref{plane}.

	\section{Rigidity result. Theorem~\ref{thm:EDO}}
	\label{sec:EDO}
	
	The purpose of this section is to prove the rigidity result stated in Theorem~\ref{thm:EDO} concerning (weak) solutions of the system
	\begin{equation}
	\label{eq:self-similar}
	\frac{x \f'}2=\beta \f\times \f''-\alpha \f\times (\f\times \f''),
	\quad\textup{ on }\R.
	\end{equation}
	We start by introducing the notion of weak solution of the above system. To this end, we first observe that the system \eqref{eq:self-similar} recasts as
	\begin{equation}
	\label{eq:EDO}
	\alpha \f''+\alpha\abs{\f'}^2\f +\beta (\f\times \f')'-\frac{x \f'}2=0,
	\end{equation}
	using  the following vector identities for a (smooth) function 
	$\f$  with $\abs{\f}=1$:
	\bq
	\label{identities}
	\begin{aligned}
		\f\times  \f''&=(\f\times  \f')',\\
		-\f\times (\f\times \f'')&= \f''+\abs{ \f'}^2\f.
	\end{aligned}
	\eq
	We prefer to use the formulation \eqref{eq:EDO} since it is simpler to handle in  weak sense.
	Indeed, we say that $\f=(f_1,f_2,f_3)\in H^1_{\loc}(\R,\S^2)$ is a \emph{weak solution} to
	the system \eqref{eq:EDO} if 
	\begin{equation}
	\label{weak:EDO}
	\int_{\R}
	\Big(-\alpha \f' \cdot \vp'
	+\alpha \abs{\f'}^2\f\cdot \vp
	-\beta(\f\times \f')\cdot \vp'
	-\frac{x}2   \f' \vp
	\Big)dx=0,
	\end{equation}
	for all $\vp=(\varphi_1,\varphi_3,\varphi_3)\in C_0^\infty(\R).$
	
	Using \eqref{identities},
	we can recast \eqref{eq:EDO} as,
	\begin{subequations}
		\label{sys:ODE}
		\begin{align}
		\label{f1}	\alpha f_1''+\alpha \abs{\f'}^2 f_1+\beta(f_2f_3''-f_3 f_2'')-\frac{x}2f'_1=0,\\
		\label{f2} \alpha f_2''+\alpha \abs{\f'}^2 f_2+\beta(f_3f_1''-f_1 f_3'')-\frac{x}2f'_2=0,\\
		\label{f3} \alpha f_3''+\alpha \abs{\f'}^2 f_3+\beta(f_1f_2''-f_2 f_1'')-\frac{x}2f'_3=0.
		\end{align}
	\end{subequations}
	Thus we see that the weak formulation \eqref{weak:EDO}  can be written 
	as 
	\bq
	\label{weak:sys}
	\int_{\R}\boA(\f(x)) \f'(x) \cdot \vp'(x)=\int_{\R} G(x,\f,\f') \vp(x), \quad 
	\text{for all }\vp \in C_0^\infty(\R),
	\eq
	with 
	$$\boA(\u)=\begin{pmatrix}
	\alpha & -\beta u_3 & \beta u_2\\
	\beta u_3 & \alpha & -\beta u_1\\
	-\beta u_2 & \beta u_1 & \alpha,
	\end{pmatrix}
	\quad 
	\textup{ and }
	\quad
	G(x,\u,\bm p)=
	\begin{pmatrix}
	\alpha u_1 \abs{\bm p}^2-\frac{xp_1}2 \\
	\alpha u_2 \abs{\bm p}^2-\frac{xp_2}2\\
	\alpha u_3 \abs{\bm p}^2-\frac{xp_3}2\\
	\end{pmatrix},
	$$
	where $\u =(u_1,u_2,u_3)$ and $\bm p=(p_1,p_2,p_3)$.
	We want now to invoke the regularity theory for quasilinear elliptic system
	(see \cite{lady,giaquinta83}). To verify that the system is indeed uniformly elliptic, 
	we can easily check that 
	$$\boA(\u) \bm \xi \cdot \bm \xi= \alpha \abs{\bm\xi}^2, \quad \text{ for all }\bm\xi,\u\in \R^3.
	$$
	In addition, $G$ has quadratic growth on bounded domains, i.e.\
	$$
	\abs{G(x,\u,\bm p)}\leq \sqrt{3}(M\abs{\bm p}^2+R\abs{\bm p}),
	$$
	for all $\abs{\bm u }\leq M$ and $\abs{x}\leq R$.
	Since a weak solution $\f$ to \eqref{weak:sys} belongs by definition 
	to $H^1_{\loc}(\R;\S^2)$, we have by the Sobolev embedding theorem 
	that $\f$ is H\"older continuous with $\abs{\f(x)}=1$.
	Therefore we can apply the results in Theorem 1.2 in \cite{giaquinta83}
	(see also Lemma~8.6.3 in \cite{jost} or Theorem~2.4.3~in \cite{schulz} for detailed proofs),
	to conclude that $\f \in H^2_{\loc}(\R)\cap W^{1,4}_{\loc}(\R)$, and so 
	that $\f\in \boC_{\loc}^{1,\gamma}(\R)$, for some $\gamma \in (0,1)$.
	We get that $G(x,\f(x),\f'(x))$ belongs to $\boC_{\loc}^{0,\gamma}(\R)$, 
	which allows us to invoke the Schauder regularity theory (see e.g.\ Theorem~A.2.3 in \cite{jost})
	to infer that $\f\in \boC_{\loc}^{2,\gamma}(\R)$. This implies that 
	$G(x,\f(x),\f'(x))$ belongs to $\boC_{\loc}^{1,\gamma}(\R)$, as well as the coefficients
	of $\boA(\u)$, so the Schauder estimates yield that $\f\in \boC_{\loc}^{3,\gamma}(\R)$.
	By induction, we this argument shows that $\f\in \boC^\infty(\R)$.
	
	We are now in position to  complete the proof of Theorem~\ref{thm:EDO}. Indeed, let first remark that differentiating  the relation $\abs{\f}^2=1$, we have
	the identities 
	\begin{gather}
	\label{idenX}
	\f\cdot \f'=0,\\
	\label{idenY}
	\f\cdot \f''=-\abs{\f'}^2.
	\end{gather}
	By taking the cross product of $\f$ and \eqref{eq:EDO}, and using \eqref{identities}, we have
	\bq
	\label{eq:EDO-new}
	\beta \f''+\beta\abs{\f'}^2\f-\alpha (\f\times \f')'+\frac{x}2 \f \times  \f'=0.
	\eq
	Thus, by multiplying  \eqref{eq:EDO} by $\alpha$, 
	\eqref{eq:EDO-new} by $\beta$, and recalling that $\alpha^2+\beta^2=1$,
	we  get
	$$
	\f''+\abs{\f'}^2\f- \frac{x}2 (\alpha \f'-\beta \f\times \f')=0.
	$$ 
	Taking the scalar product of this equation and $\f'$, the identity \eqref{idenX} allow us to conclude that 
	\bq
	\frac12 (\abs{\f'}^2)'-\frac{\alpha x}2  \abs{\f'}^2=0. 
	\eq
	Integrating, we deduce that there is a constant $C\geq 0$ such that 
	$\abs{\f'}^2=Ce^{\alpha x^2/2}$. This completes the proof of Theorem~\ref{thm:EDO}.
	
	\bigskip
	
	We conclude this section with some remarks.
	\begin{remark}
		A similar result to the one stated in Theorem~\ref{thm:EDO} also holds for the expanders solutions. Precisely, any weak solution to 
		\eqref{eq:self-similar}, with $x\f'/2$ replaced by $-x\f'/2$ in the l.h.s.,
		is smooth and  there exists $c\geq 0$ such that  $\abs{\f'(x)}=ce^{-\alpha x^2/4}$, for all $x\in \R$. 
	\end{remark}
	\begin{remark}
		Let us mention that in the case $\alpha=1$, a nonconstant solution $\u : \R^N\to \S^d$
		to equation 
		\bq
		\label{quasi-harmonic}
		\Delta \u+\abs{\grad{\u}}^2\u -\frac{x \cdot \grad \u}2 =0, 
		\quad\textup{ on }\R^N,
		\eq
		is usually called {\em quasi-harmonic sphere}, since it corresponds to the Euler--Lagrange 
		equations of a critical point of the (so-called) quasi-energy \cite{lin-wang} $$E_{\textup{quasi}}(\u)=\int_{\R^N}\abs{\grad{\u(y)}}^2e^{-\abs{y^2}/4}dy.$$
		It has been proved in \cite{fan} the existence of a (real-valued) function $h$
		such that 
		$$\u(x)=\Big(\frac{x}{\abs{x}} \sin(h(\abs{x})),\cos(h(\abs{x})\Big)
		$$
		is a solution to \eqref{quasi-harmonic}  with finite quasi-energy for $3\leq N=d\leq 6$.
		In addition, there is no solution of this form if $d\geq 7$ \cite{bizon-wasserman}.
		Both results are based on the analysis of the second-order ODE associated with $h$. 
		We refer also to \cite{gastel} for a generalization of the existence result for $N\geq 3$
		of other equivariant solutions  to \eqref{quasi-harmonic}.
		In the case $N=1$ and $d=2$, the solution to \eqref{quasi-harmonic} is  explicitly given by \eqref{solution-alpha-1},
		and its associated quasi-energy is infinity, as remarked in \cite{xu-zhou}.
	\end{remark}
	
	\section{Existence, uniqueness and properties}
	\label{sec:existence}
	
	\subsection{Existence and uniqueness of the self-similar profile. Proposition~\ref{prop:unicidad}}
	\label{subsec:existence}
	In the previous section we have shown that any solution to the profile equation
	\begin{equation}
	\label{profile}
	\alpha \m''+\alpha\abs{\m'}^2\m +\beta (\m\times \m')'-\frac{x \m'}2=0,
	\end{equation}
	is smooth and that there is $c\geq 0$ such that
	\bq
	\label{derivada}
	\abs{\m'(x)}=c e^{\alpha x^2/4},\quad \text{ for all }x\in \R.
	\eq
	We want to give now the details about how to construct such a solution by using the Serret--Frenet frame, which will correspond to the profile $\m_{c,\alpha}$ in Theorem~\ref{thm-conv}. 
	The idea is to identify $\m$ as the tangent vector to a curve in $\R^3$, so we  
	first  recall some facts about curves in the space. 
	Given   $\m : \R\to \S^2$ a smooth function, we can define the curve 
	\bq
	\label{def:curve}
	\bm X_{\m}(x)=\int_0^x \m(s)ds,
	\eq
	so that $\bm X_{\m}$ is  smooth, parametrized by arclenght,
	and its  tangent vector is $\m$. In addition, if $\abs{\m'}$ does not vanish on $\R$, 
	we can define the normal vector 
	${\n(x)}={\m'(x)}/{\abs{\m'(x)}}$ and  
	the binormal vector $\b(x)=\m(x)\times \n(x)$.
	Moreover, we can define the curvature and torsion of $\bm X_{\m}$ 
	as $k(x)={\abs{\m'(x)}}$ and $\tau(x)=-{\b'(x)\cdot \n(x)}$.
	Since $\abs{\m(x)}^2=1,$ for all $x\in\R$, we have that $\m(x)\cdot \n(x)=0$, for all $x\in \R$,
	that the vectors $\{\m,\n,\b\}$ are orthonormal and it is standard to check that they satisfy the Serret--Frenet system
	\begin{equation}\label{profile-serret}
	\left\{
	\begin{array}{ll}
	\m'&=k\n,\\
	\n'&=-k\m +\tau \b,\\
	\b'&=-\tau \n.
	\end{array}
	\right.
	\end{equation}
	Let us apply this construction to find a solution to \eqref{profile}.
	We define curve $\bm X_{\m}$ 
	as in \eqref{def:curve}, and remark  that 
	equation \eqref{profile} rewrites in terms of $\{\m, \n , \b   \}$ as 
	
	$$
	\frac{x}{2}k\n=\beta(k'\b-\tau k\n )-\alpha(-k'\n-k\tau\b).
	$$
	Therefore, from the orthogonality of the vectors $\n$ and $\b$,  we conclude that 
	the curvature and torsion of $\bm X_{\m}$ are solutions of the equations
	$$
	\frac{x}{2} k=\alpha k'-\beta \tau k\qquad {\hbox{and}}\qquad \beta k'+\alpha k\tau=0,
	$$
	that is
	\begin{equation}
	\label{profile-kt}
	k(x)= c e^{\frac{\alpha x^2}{4}}\qquad {\hbox{and}}\qquad \tau(x)=-\frac{\beta x}{2},
	\end{equation}
	for some $c\geq 0$. Of course, the fact that  $k(x)= c e^{{\alpha x^2}/{4}}$
	is in agreement with the fact that we must  have $\abs{\m'(x)}=c e^{\alpha x^2/4}$.
	
	Now, given $\alpha\in[0,1]$ and $c>0$, consider the Serret--Frenet system \eqref{profile-serret} with curvature and torsion function given by \eqref{profile-kt}
	and initial conditions
	\begin{equation}\label{profile-initial}
	{\m}(0)=(1,0,0), \quad
	{\n}(0)=(0,1,0), \quad {\b}(0)=(0,0,1).
	\end{equation}
	Then, by standard ODE theory, there exists a unique global solution $\{\m_{c,\alpha}, \n_{c,\alpha}, \b_{c,\alpha}\}$ in $(\mathcal{C}^{\infty}(\mathbb{R}; \mathbb{S}^2))^{3}$, and these vectors are orthonormal. Also, it is straightforward to verify that $\m_{c,\alpha}$
	is a solution to \eqref{profile} satisfying  
	\eqref{derivada}.
	
	The above argument provides the existence of solutions 
	in the statement of Proposition~\ref{prop:unicidad}.
	We will now complete the proof of Proposition~\ref{prop:unicidad} showing the uniqueness of
	such  solutions, up to rotations.

	To this end,  assume that $\tilde \m$ is a weak nontrivial  solution to \eqref{profile}.
	By Theorem~\ref{thm:EDO}, $\tilde \m$ is in $\boC^{\infty}(\R,\S^2)$ and 
	there exists $c>0$  such that 
	$\abs{\tilde \m'(x)}=c e^{\alpha x^2/4}$, for all $x\in \R$.
	Following the above argument, the  curve $\bm X_{\tilde \m}$ (defined  in \eqref{def:curve}), has curvature  $c e^{{\alpha x^2}/{4}}$ and  torsion  $-\beta x/2$.
	Since the curve $\bm X_{\m_{c,\alpha}}$  associated with $\m_{c,\alpha}$,
	and $\bm X_{\tilde \m}$ have the same curvature and torsion, using 
	fundamental theorem of the local theory of space curves (see e.g. Theorem~1.3.5 in \cite{montiel-ros}),
	we conclude  
	that both curves are equal up to direct rigid motion, i.e.\ 
	there exist $\bm p\in\R^3$ and $\boR \in SO(3)$ such that 
	$\bm X_{\tilde \m}(x)=\boR(\bm X_{\m_{c,\alpha}}(x))+\bm p$, for all $x\in \R^3$.
	By differentiating this identity, we finally get that $\tilde \m=\boR \m_{c,\alpha}$,
	which proves the uniqueness of solution, up to a rotation, as stated in Proposition~\ref{prop:unicidad}.
	
	\subsection{Asymptotics of the self-similar profile}
	The rest of this section is devoted to establish  properties of the family of solutions $\{ \mm_{c, \alpha}\}_{c,\alpha}$,
	for fixed $\alpha\in(0,1]$ and $c>0$. Due to the self-similar nature of these solutions, this analysis reduces to study the properties of the associated profile $\m_{c,\alpha}$, or equivalently, of the solution $\{ \m_{c,\alpha}, \n_{c,\alpha}, \b_{c,\alpha}  \}$ of the Serret--Frenet system \eqref{profile-serret} with curvature and torsion given in \eqref{profile-kt}, and initial conditions \eqref{profile-initial}.

	It is important to mention that the recovery of the properties of the trihedron  $\{\m, \n,\b\}$, and in particular of the profile $\m$, from the knowledge of its curvature and torsion  is a difficult question. This can be seen from the equivalent formulations 
	of the Serret--Frenet equation in terms of a second-order complex-valued highly non-linear EDO,
	or in terms of a complex-valued Riccati equation (see e.g.\ \cite{darboux,struik,lamb,gutierrez-delaire1}). For this reason, the integration 
	of the trihedron can  often only be done numerically, rather than analytically.

	Since the Serret--Frenet equations are decoupled, we start by analyzing
	the system for the  {\em scalar} functions $m_{c,\alpha}$, $n_{c,\alpha}$ and $b_{c,\alpha}$
	\begin{equation}\label{serret2}
	\left\{
	\begin{aligned}
	\gm_{c,\alpha}'(x)&=c e^{\frac{\alpha x^2}{4}}\gn_{c,\alpha}(x),\\
	\gn_{c,\alpha}'(x)&=-ce^{\frac{\alpha x^2}{4}}\gm_{c,\alpha}(x) -\frac{\beta x}{2} \gb_{c,\alpha}(x),\\
	\gb_{c,\alpha}'(x)&=\frac{\beta x}{2} \gn_{c,\alpha}(x),
	\end{aligned}
	\right.
	\end{equation}
	with initial conditions $(\gm_{c,\alpha},\gb_{c,\alpha}, \gn_{c,\alpha})(0)$, that we suppose independent of $c$ and $\alpha$, and  satisfying
	$$
	\gm_{c,\alpha}(0)^2+\gb_{c,\alpha}(0)^2+\gn_{c,\alpha}(0)^2=1.
	$$
	Then by ODE theory, the solution is smooth, global and satisfies
	\begin{equation}
	\label{unidad}
	\gm_{c,\alpha}(x)^2+\gb_{c,\alpha}(x)^2+\gn_{c,\alpha}(x)^2=1, \quad \text{for all }x\in \R.
	\end{equation}
	Moreover, the solution  depends continuously on the parameters $c>0$ and $\alpha\in (0,1]$.
	
	To study the behavior of the solution of the system \eqref{serret2}, we need  some elementary bounds for the non-normalized complementary error function.
	\begin{lemma}
		\label{lemma-bounds-erfi}
		Let $\gamma\in(0,1]$. The following upper bounds hold for $x>0$
		\begin{equation}\label{erfc1}
		\int_{x}^\infty e^{-\gamma s^2}ds\leq \frac{1}{2\gamma x}e^{-\gamma x^2}
		\quad\text{ and }\quad
		\int_{x}^\infty se^{-\gamma s^2}ds= \frac{1}{2\gamma}e^{-\gamma x^2}.
		\end{equation}
		Also, for $\gamma\in (0,1]$ and $x\geq 1$,
		\begin{equation}\label{erfc2}
		\int_{x}^\infty s^2 e^{-\gamma s^2}ds\leq \frac{x}{\gamma^2}e^{-\gamma x^2},\quad\text{ and }\quad
		\int_{x}^\infty s^3 e^{-\gamma s^2}ds\leq \frac{x^2}{\gamma^2}e^{-\gamma x^2}.
		\end{equation}
	\end{lemma}
	\begin{proof}
		We start recalling some standard bounds the complementary error function (see e.g.~\cite{abram,gordon})
		\begin{gather}
		\label{ineq-erfc2}\frac{xe^{-x^2}}{2x^2+1}\leq \int_x^\infty  e^{- s^2}ds \leq \frac{e^{-x^2}}{2x}, \quad \text{  for }x> 0.
		\end{gather}
		The first formula in \eqref{erfc1} follows by scaling this inequality.
		The second formula in \eqref{erfc1} follows by integration by parts.
		
		To prove the first estimate in \eqref{erfc2}, we use  integration by parts and  \eqref{erfc1} to show that 
		$$
		\int_{x}^\infty s^2 e^{- \gamma s^2}ds=\frac{xe^{-\gamma x^2}}{2\gamma}+\frac{1}{2\gamma}\int_{x}^\infty  e^{-\gamma s^2}\, ds\leq e^{- \gamma x^2}\left(\frac{x}{2\gamma}+\frac{1}{4\gamma^2 x}\right)\leq
		xe^{- \gamma x^2}\left(\frac{1}{2\gamma}+\frac{1}{4\gamma^2 }\right),
		\quad {{\forall x\geq 1}}.
		$$
		Since $\gamma\in (0,1]$,  we have $\gamma^2\leq \gamma$ and thus we conclude the estimate for the desired integral.
		The second inequality in \eqref{erfc2} easily follows from the identity
		$$\int_x^\infty s^3 e^{-\gamma s^2}ds=\frac{1+\gamma x^2}{2\gamma^2}e^{- \gamma x^2}, \qquad \forall x\in\mathbb{R},
		$$
		noticing that $1+\gamma x^2\leq x^2(1+\gamma)\leq 2x^2$, since $x\geq 1$ and $\gamma\in (0,1]$.
	\end{proof}
	Now we can state a first result on the behavior of  $\{\gm_{c,\alpha}, \gn_{c,\alpha}, \gb_{c,\alpha}\}$.
	\begin{prop}
		\label{mnb-formulae}
		Let $\alpha\in(0,1]$ and $c>0$, and define
		$$
		\Phi_\alpha(x)=\int_{0}^{x} e^{\frac{\alpha s^2}{4}}\, ds. 
		$$
		Then the following statements hold.
		\begin{itemize}
			\item[\it{i)}] For all $x \in \R$,
			\begin{equation}\label{b-asymp}
			\gb_{c,\alpha}(x)=B_{c,\alpha}+\frac{\beta x}{2c} e^{-\alpha x^2/4 }\gm_{c,\alpha}(x)+
			\frac{\beta}{2c}\int_x^\infty \left(1-\frac{\alpha s^2}{2} \right)e^{-\alpha s^2/4 }\gm_{c,\alpha}(s)ds,
			\end{equation}
			where
			\begin{equation}\label{b-infty}
			B_{c,\alpha}=\gb_{c,\alpha}(0)-\frac{\beta}{2c}\int_0^\infty \left(1-\frac{\alpha s^2}{2} \right)e^{-\alpha s^2/4 }\gm_{c,\alpha}(s)ds.
			\end{equation}
			In particular, for all $x\geq 1$
			\begin{equation}\label{est-b}
			\abs{\gb_{c,\alpha}(x)-B_{c,\alpha}}\leq \frac{6\beta }{c{\alpha}}x e^{-\alpha x^2/4 }.
			\end{equation}
			\item[\it{ii)}] Setting $\gw_{c,\alpha}=\gm_{c,\alpha}+i\gn_{c,\alpha}$,   for all $x \in \R$, we have
			\begin{multline}\label{w-asymp}
			\gw_{c,\alpha}(x)=
			e^{-ic\, \Phi_\alpha(x)}
			\Big(
			W_{c,\alpha}-\frac{\beta x}{2c} e^{ic\, \Phi_\alpha(x)-\alpha x^2/4 }\gb_{c,\alpha}(x)
			\\
			-\frac{\beta}{2c}\int_x^\infty e^{ic\, \Phi_\alpha(s)-\alpha s^2/4 }
			\big(
			\frac{\beta s^2}2\gn_{c,\alpha}(s)+\big(1-\frac{\alpha s^2}{2}\big) \gb_{c,\alpha}(s)
			\big) ds
			\Big),
			\end{multline}
			where
			\begin{equation}\label{w-infty}
			W_{c,\alpha}=\gw_{c,\alpha}(0)+\frac{\beta}{2c}\int_0^\infty e^{ic\, \Phi_\alpha(s)-\alpha s^2/4 }
			\big(
			\frac{\beta s^2}2\gn_{c,\alpha}(s)+\big(1-\frac{\alpha s^2}{2}\big) \gb_{c,\alpha}(s)
			\big) ds.
			\end{equation}
			In particular, for all $x\geq 1$,
			\begin{equation}\label{est-w}
			\abs{\gw_{c,\alpha}(x)- e^{-ic\, \Phi_\alpha(x)}W_{c,\alpha}}\leq \frac{10\beta }{c\alpha^2}x e^{-\alpha x^2/4 }.
			\end{equation}
		\end{itemize}
		Furthermore, the limiting values $B_{c,\alpha}$ and $W_{c,\alpha}$ are  separately continuous
		functions of $(c,\alpha)$ for $(c,\alpha)\in (0,\infty)\times (0,1]$.
	\end{prop}
	\begin{proof}
		For simplicity, we will drop the subscripts $c$ and $\alpha$ if there is no possible confusion.  From \eqref{serret2}, we get
		\bq
		\label{eq-b}
		\begin{aligned}
			\gb(x)-\gb(0)=&\int_0^x \gb'(s)ds=\frac{\beta}{2c} \int_0^x se^{-\frac{\alpha s^2}{4}}\gm'(s)ds
			\\
			=& \frac{\beta}{2c}\Big(
			xe^{-\frac{\alpha x^2}{4}}\gm(x)- \int_0^x \big(1-\frac{\alpha s^2}{2}\big) e^{-\frac{\alpha s^2}{4}}\gm(s)ds\Big),
		\end{aligned}
		\eq
		where we have used integration by parts.
		Notice that $\int_{0}^{\infty} (1-\alpha s^2/2)e^{-\alpha s^2/4}\gm(s)\, ds$ is well-defined, since $\alpha\in(0,1]$ and $\gm$ is bounded. Therefore, the existence of $B:=\lim_{x\rightarrow\infty}\gb(x)$ follows from \eqref{eq-b}. Moreover,
		$$
		B:=\gb(0)-\frac{\beta}{2c}\int_0^\infty \left(1-\frac{\alpha s^2}{2} \right)e^{-\alpha s^2/4 }\gm(s)ds.
		$$
		Formula \eqref{b-asymp} easily follows from integrating $\gb'$ from $x\in\mathbb{R}$ to $\infty$ and arguing as above.
		
		To prove \eqref{est-b}, it is enough to observe that by Lemma~\ref{lemma-bounds-erfi}, for $x\geq1$ and $0<\alpha\leq 1$,
		\begin{equation}\label{est-integral}
		\int_x^\infty e^{-\alpha s^2/4 }ds\leq
		\frac{2}{\alpha x} e^{-\frac{\alpha x^2}{4}} 
		\leq \frac{2}{\alpha}x e^{-\frac{\alpha x^2}{4}}, 
		\quad \text{ and }\quad
		\int_x^\infty s^2 e^{-\alpha s^2/4 }ds\leq
		\frac{16}{\alpha^2}xe^{-\frac{\alpha x^2}{4}}.
		\end{equation}
		
		Setting   $\gw=\gm+i\gn$ and using \eqref{serret2}, we obtain that $\gw$ satisfies the ODE
		\begin{equation}\label{eq-w}
		\gw'+i c e^{\alpha x^2/4}\gw= - i\frac{\beta x}{2} \gb(x),
		\end{equation}
		or, equivalently,
		\begin{equation}
		\label{eq-w1}
		\Big( e^{ic\Phi_\alpha(x)} \gw \Big)^{'}= - i\frac{\beta x}{2} \gb(x)e^{ic\Phi_\alpha(x)}.
		\end{equation}
		Integrating \eqref{eq-w1} from $0$ to $x>0$, and writing
		$$
		e^{ic\Phi_\alpha(x)}=-\frac{i}{c} \Big( e^{ic\Phi_\alpha(x)}\Big)^{'} e^{-\alpha x^2/4},
		$$
		integrating by parts, and using once again  \eqref{serret2}, we get
		\begin{align*}
		e^{ic\Phi_\alpha(x)}\gw(x)=&
		\gw(0)-\frac{\beta}{2c}x\gb(x)e^{ic\Phi_\alpha(x)-\alpha x^2/4}
		\\
		&+\frac{\beta}{2c}\int_{0}^{x}  e^{ic\Phi_\alpha(s)-\alpha s^2/4} \Big(  \frac{\beta}{2}s^2\gn(s)+(1-\alpha\frac{s^2}{2})\gb(s)  \Big)\, ds.
		\end{align*}
		Since $\alpha\in(0,1]$, from the above identity it follows the existence of
		$$
		W:=\lim_{x\rightarrow \infty} e^{ic\Phi_\alpha(x)} \gw(x),
		$$
		and formula \eqref{w-infty} for $W$.
		
		Formula \eqref{w-asymp} now follows from integrating \eqref{eq-w1} from $x>0$ to $\infty$ and arguing as in the previous lines. The estimate in \eqref{est-w} can be deduced as before, 
		since the bounds in \eqref{est-integral} imply that 
		$$
		\abs{\gw_{c,\alpha}(x)- e^{-ic\, \Phi_\alpha(x)}W_{c,\alpha}}\leq \frac{\beta }{2c}x e^{-\alpha x^2/4 }\left(1+\frac{16(\alpha+\beta)}{2\alpha^2}+\frac{2}{\alpha}\right)
		\leq  \frac{10\beta }{c\alpha^2}x e^{-\alpha x^2/4 },
		$$
		where we used that $\alpha+\beta\leq 2$ and $\alpha \leq 1$.
		
		To see that the limiting values $B_{c,\alpha}$ and $W_{c,\alpha}$ given by \eqref{b-infty} and \eqref{w-infty} are continuous functions of  $(c,\alpha)$, for $(c,\alpha)\in(0,\infty)\times(0,1]$, we recall that by standard ODE theory, the functions
		$\gm_{c,\alpha}(x)$, $\gn_{c,\alpha}(x)$ and $\gb_{c,\alpha}(x)$ are continuous functions of $x$, $c$ and $\alpha$.
		Then, the dominated convergence theorem applied to the formulae \eqref{b-infty} and \eqref{w-infty} yield the desired continuity.
	\end{proof}
	\begin{remark}
		\label{rem-alpha-1}
		As mentioned before, the shrinkers of the 1d-harmonic heat flow can be computed explicitly,
		because  if $\alpha=1$, the system \eqref{curvature}-\eqref{serret}-\eqref{IC} can be solved easily. Indeed, in this case $\beta=0$, so that 
		we obtain	
		\begin{align*}
		\m_{c,1}(x)&=(\cos(c\Phi_1(x)),\sin(c\Phi_1(x)), 0),\\
		\n_{c,1}(x)&=(-\sin(c\Phi_1(x)),\cos(c\Phi_1(x)), 0),\\
		\b_{c,1}(x)&=(0,0,1),
		\end{align*}
		for all $x\in \R$. 	
	\end{remark}
	In order to obtain a better understanding of the asymptotic behavior of $\{\gm_{c,\alpha}, \gn_{c,\alpha}, \gb_{c,\alpha}\}$, we need to exploit the osci\-lla\-tory character of the function
	$e^{ic \Phi_{\alpha}(s)}$ in the integrals \eqref{b-asymp} and \eqref{w-asymp}. In our arguments we will use the following two lemmas.
	\begin{lemma}
		\label{lem-osc1}
		Let $0<\alpha\leq 1$. For $\sigma\in \R\setminus\{0\}$ and $x\in \R$, the limit
		$$ \int_x^\infty se^{i\sigma \Phi_\alpha(s)}ds:= \lim_{y\to \infty}  \int_x^y se^{i\sigma \Phi_\alpha(s)}ds$$
		exists. Moreover, for all $x\geq 1$,
		\begin{equation}
		\label{est-osc1}
		\left| \int_x^\infty se^{i\sigma \Phi_\alpha(s)}ds \right|\leq 
		\frac{11x}{|\sigma|\alpha}e^{-\alpha x^2/4},
		\end{equation}
		and
		\begin{equation}
		\label{est-osc2}
		\int_x^\infty se^{i\sigma \Phi_\alpha(s)}ds =i\frac{x}{\sigma}e^{i\sigma \Phi_\alpha(x)-\alpha x^2/4}+
		\boO\left( \frac{{{x^2}}}{\sigma^2}e^{-\alpha x^2/2}  \right).
		\end{equation}
	\end{lemma}
	\begin{proof}
		Let $x\in\mathbb{R}$ and take $y\geq x$. Then, integrating by parts,
		\begin{align}
		\label{est-osc3}
		\int_x^y se^{i\sigma \Phi_\alpha(s)}ds=&
		\frac{1}{i\sigma}\int_x^y s(e^{i\sigma \Phi_{\alpha}(s)})'e^{-\alpha s^2/4}\, ds
		\nonumber \\
		=&\left. \frac{s}{{i\sigma}} e^{i\sigma \Phi_{\alpha}(s)-\alpha s^2/4}\right|_{x}^y-
		\frac{1}{i\sigma}\int_{x}^y e^{i\sigma \Phi_{\alpha}(s)-\alpha s^2/4}\big(1-\frac{\alpha {{s^2}}}{2}\big)\, ds.
		\end{align}
		The existence of the improper integral $\int_{x}^{\infty} se^{i\sigma \Phi_{\alpha}(s)}\, ds$ follows taking the limit as $y$ goes to $\infty$ in the above formula,
		and bearing in mind that $\alpha>0$. The estimate \eqref{est-osc1} follows from \eqref{est-integral} and the fact that $x\geq 1$ and $0<\alpha\leq 1$. Finally, integrating by parts once more, we have
		\begin{equation*}
		i\sigma\int_{x}^\infty e^{i\sigma \Phi_{\alpha}(s)-\alpha s^2/4}\big(1-\frac{\alpha {{s^2}}}{2}\big) ds
		=
		-e^{i\sigma \Phi_{\alpha}(x)-\alpha x^2/2}\big(1-\frac{\alpha {{x^2}}}{2}\big)
		-
		\int_{x}^\infty e^{i\sigma \Phi_{\alpha}(s)-\alpha s^2/2}\big(\frac{\alpha^2 {{s^3}}}{2}
		-2\alpha s
		\big) ds.
		\end{equation*}
		Hence,  using Lemma~\ref{lemma-bounds-erfi} and \eqref{est-osc3}, we obtain \eqref{est-osc2}.
	\end{proof}
	\begin{lemma}
		\label{lem-osc2}
		Let $\sigma\in \R\setminus\{0\}$, $\gamma\in \R$, $\alpha>0$ and set $\tilde \gamma=\gamma+\alpha/4$. If ${{0<\tilde\gamma\leq 1}}$, then for $x\geq 1$,
		\begin{equation}
		\begin{gathered}
		\int_x^\infty e^{i\sigma \Phi_\alpha(s)-\gamma s^2}ds={{\boO\left( \frac{e^{-\tilde \gamma x^2}}{|\sigma| }  \right)}}, \qquad
		\int_x^\infty se^{i\sigma \Phi_\alpha(s)-\gamma s^2}ds={{\boO\left( \frac{x e^{-\tilde \gamma x^2}}{|\sigma|{\tilde\gamma}})  \right) }},\\
		\int_x^\infty s^2e^{i\sigma \Phi_\alpha(s)-\gamma s^2}ds={{\boO\left( \frac{x^2 e^{-\tilde \gamma x^2}}{|\sigma|\tilde\gamma}  \right)   }}.
		\end{gathered}
		\end{equation}
	\end{lemma}
	\begin{proof}
		For $n\in\{0,1,2\}$, we set 
		$$I_n=\int_x^\infty s^n e^{i\sigma \Phi_\alpha(s)-\gamma s^2}ds.$$
		$$
		I_n=\frac{1}{i\sigma}\left( 
		-x^ne^{i\sigma \Phi_\alpha(x)-\tilde\gamma x^2}  -\int_x^\infty  e^{i\sigma \Phi_\alpha(s)-\tilde\gamma s^2}\left( 
		ns^{n-1}-2  \tilde\gamma s^{n+1 }
		\right) ds
		\right).
		$$
		Then the desired asymptotics follow from Lemma~\ref{lemma-bounds-erfi}.
	\end{proof}
	Using previous lemmas, we can now improve the asymptotics in Proposition~\ref{mnb-formulae}
	and obtain explicitly the term decaying as  $e^{-\alpha x^2/4}$ (multiplied by a polynomial).
	\begin{cor}
		\label{cor-asymp}
		With the same notation as in Proposition~\ref{mnb-formulae}, the following asymptotics hold for $x\geq 1$
		\begin{align}
		\label{asymp-b}
		\gb_{c,\alpha}(x)=&B_{c,\alpha}+\frac{\beta x}{2c} e^{-\alpha x^2/4}
		\Re (e^{-ic\Phi_\alpha(x)} W_{c,\alpha})
		+\frac{\beta }{c^2 \alpha^{3}}  \boO(x^2 e^{-\alpha x^2/2}),
		\\
		\label{asymp-w}
		\gw_{c,\alpha}(x)=&
		e^{-ic\, \Phi_\alpha(x)}
		\Big(
		W_{c,\alpha}-\frac{\beta B_{c,\alpha}}{2c}xe^{ic\, \Phi_{\alpha}(x)-\alpha x^2/4}+\frac{i\beta^2 W_{c,\alpha}}{8c}\int_{x}^{\infty}s^2e^{-\alpha s^2/4}\, ds
		\Big)\\
		&+ \frac{\beta}{c^2 \alpha^5}\boO(x^2 e^{-\alpha x^2/2}).\nonumber
		\end{align}
	\end{cor}
	\begin{proof}
		As usual, we drop the  subscripts $c$ and $\alpha$ in  the rest of the proof.
		Recalling that $\gw=\gm+i\gn$, we have from \eqref{est-w},
		$$
		\gm=\Re (e^{-ic\Phi_\alpha(x)}  W)+ \frac{\beta}{c\alpha^2}
		\boO(x e^{-\alpha x^2/4}).
		$$
		Thus, replacing in \eqref{b-asymp}, 
		\begin{align}
		\label{eq-est-b}
		\gb(x)=&B+\frac{\beta x}{2c} e^{-\alpha x^2/4}
		\Re (e^{-ic\Phi_\alpha(x)} W)
		+\frac{\beta^2 }{c^2 \alpha^{2}}  \boO(x^2 e^{-\alpha x^2/2})
		+\boR_\gb(x),
		\end{align}
		with 
		$$
		\boR_\gb(x)=\frac{\beta}{2c}\Re\left(
		W
		\int_x^\infty \big(1-\frac{\alpha s^2}2 \big)e^{-i c\Phi_{\alpha}(s) -\alpha s^2/4}ds
		+
		\int_x^\infty \big(1-\frac{\alpha s^2}2 \big)\boO \big(
		\frac{se^{-\alpha s^2/2}}{c \alpha^2}
		\big)ds
		\right).
		$$
		By using Lemmas~ \ref{lemma-bounds-erfi} and \ref{lem-osc2} 
		to estimate the first and second integrals, respectively, we conclude that 
		\begin{equation}
		\label{eq-est-b2}
		\boR_\gb(x)=\frac{\beta}{c^2\alpha^3}\boO \big(
		x^2 e^{-\alpha x^2/2}
		\big).
		\end{equation} 
		By putting together \eqref{eq-est-b} and \eqref{eq-est-b2}, we obtain \eqref{asymp-b}.
		To establish \eqref{asymp-w} we integrate \eqref{eq-w1} from $x\geq 1$ and $\infty$,
		and use \eqref{asymp-b} and Lemma~\ref{lemma-bounds-erfi}
		to get
		\begin{equation}
		e^{ic\Phi}\gw(x)-W=I_1(x)+I_2(x)+I_3(x)+\frac{\beta^2}{c^2\alpha^5}
		\boO(xe^{-\alpha x^2/2}),
		\end{equation}
		with
		\begin{equation*}
		I_1(x)=i\frac{\beta B}{2} \int_x^{\infty} se^{ic\Phi_{\alpha}(s)}\, ds, \quad
		I_2(x)=i\frac{\beta^2 W}{8c} \int_x^\infty s^2 e^{-\alpha s^2/4}\, ds, \qquad {\hbox{and}}
		\end{equation*}
		\begin{equation*}
		I_3(x)=i\frac{\beta^2 \bar W }{8c}\int_x^\infty s^2 e^{2i c\Phi_{\alpha}(s) -\alpha s^2/4}\, ds,
		\end{equation*}
		where we have used that $\Re(z)=(z+\bar z)/2$. The conclusion follows invoking again Lemmas~\ref{lemma-bounds-erfi}, \ref{lem-osc1} and \ref{lem-osc2}.
	\end{proof}
	In Figure~\ref{fig-coordenadas} we depict the first components of the trihedron $\{  \m_{c,\alpha}, \n_{c,\alpha} , \b_{c,\alpha} \}$ for $c=0.5$ and $\alpha=0.5$, and $x>0$. As described in Corollary~\ref{cor-asymp} (recall that  $\gw_{c,\alpha}=\gm_{c,\alpha}+i\gn_{c,\alpha}$), in the plots in Figure~\ref{fig-coordenadas} one can observe that, while both $m_{1,c,\alpha}$ and $b_{1,c,\alpha}$ oscillate highly for large values of $x>0$, the component $b_{1,c,\alpha}$ converges to a limit $B_{1,c,\alpha}\approx -0.72$ as $x\rightarrow +\infty$.
	\begin{figure}[h]
		\begin{subfigure}[b]{0.33\textwidth}
			\centering
			\includegraphics[width=1\textwidth,height=0.15\textheight]{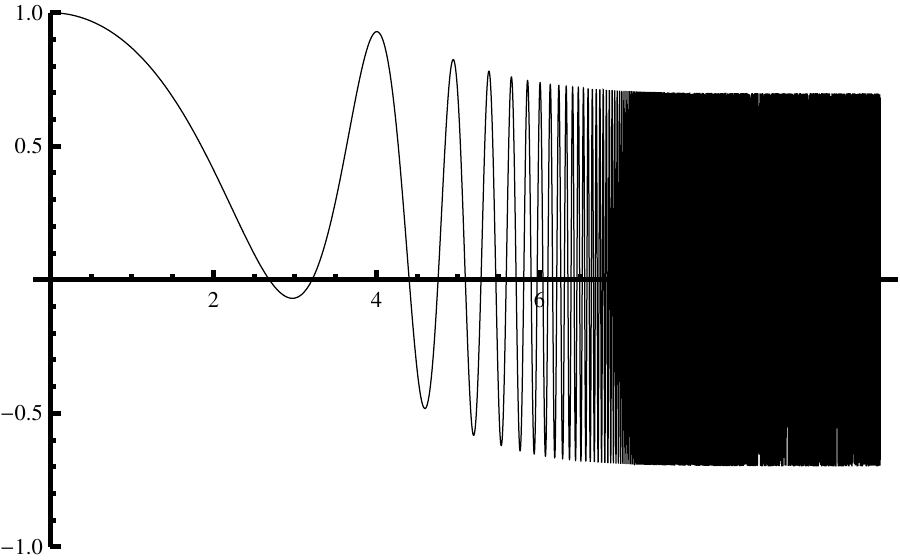}
			\caption{$m_{1,c,\alpha}$}
		\end{subfigure}
		\begin{subfigure}[b]{0.33\textwidth}
			\centering
			\includegraphics[width=\textwidth,height=0.15\textheight]{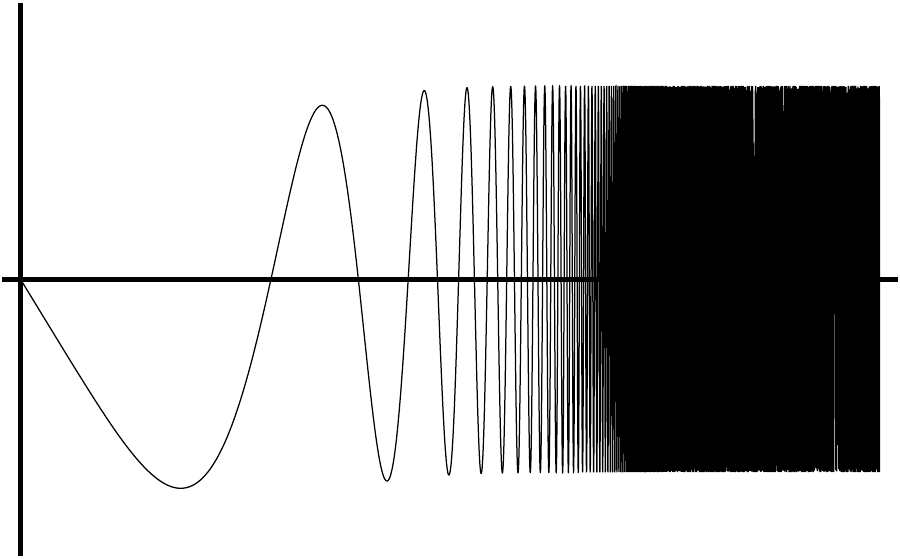}
			\caption{$n_{1,c,\alpha}$}
		\end{subfigure}
		\begin{subfigure}[b]{0.33\textwidth}
			\centering
			\includegraphics[width=0.9\textwidth,height=0.15\textheight]{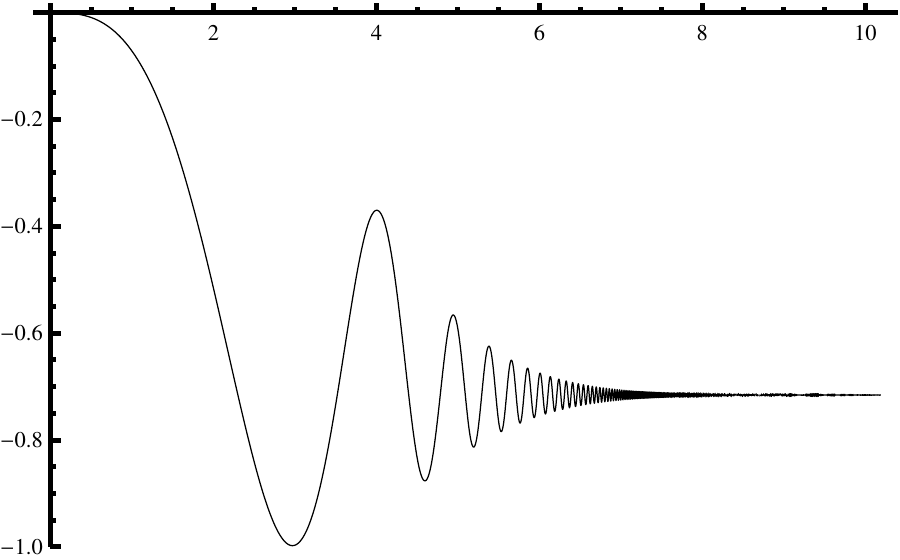}
			\caption{$b_{1,c,\alpha}$}
			\label{fig-b}
		\end{subfigure}
		\caption{Functions $m_{1,c,\alpha}$,
			$n_{1,c,\alpha}$ and $b_{1,c,\alpha}$ for $c=0.5$ and $\alpha=0.5$ on $\R^+$.
			The limit at infinity in  (iii) is  $B_{1,c,\alpha}\approx-0.72$.}
		\label{fig-coordenadas}
	\end{figure}
	\subsection{Proof of Theorem~\ref{thm-conv}}
	For simplicity, we will drop the subscripts $c$ and $\alpha$ in the proof of Theorem~\ref{thm-conv}.
	\begin{proof}[Proof of Theorem~\ref{thm-conv}]
		Let $\{ \m, \n, \b \}$ be the  solution of the Serret--Frenet system \eqref{profile-serret}--\eqref{profile-kt}  with initial condition \eqref{profile-initial}. By ODE theory, we have that the solution  $\{ \m, \n, \b \}$ is smooth, is global, and satisfies 
		\begin{equation}
		\label{unitary}
		\abs{\m(x)}=\abs{\n(x)}=\abs{\b(x)}=1, \quad \text{ for all }x\in \R,
		\end{equation}
		and  the orthogonality relations
		\begin{equation}
		\label{ortho}
		\m(x)\cdot \n(x)=\m(x)\cdot \b(x)=\n(x)\cdot \b(x)=0, \quad \text{ for all }x\in \R.
		\end{equation}
		Define
		$$
		\mm(x,t)=\m\left(  \frac{x}{\sqrt{T-t}}\right), \qquad t<T.
		$$
		Then, it is straightforward to check that $\m$ is a smooth solution of the profile equation \eqref{profile}, and consequently $\mm$ solves \eqref{LLG} for $t\in(-\infty, T)$. Moreover, using the Serret--Frenet system \eqref{profile-serret}--\eqref{profile-kt}, we have 
		$$
		|\partial_x \mm (x,t)|=\left|  \frac{1}{\sqrt{T-t}} \m'\left(  \frac{x}{\sqrt{T-t}}\right) \right|=\frac{c}{\sqrt{T-t}} e^{\frac{\alpha x^2}{4(T-t)}}.
		$$
		This shows part \ref{regular}.
		
		Notice also that, since $\{ m_j, n_j , b_j   \}$ for $j\in\{ 1,2,3\}$ solves \eqref{serret2} with initial condition
		$(1,0,0)$, $(0,1,0)$ and $(0,0,1)$ respectively, the relation \eqref{unidad} is satisfied, i.e.
		\begin{equation}\label{relation}
		m_j^2(x)+n_j^2(x)+b_j^2(x)=1, \quad \text{ for all }x\in\R, \quad j=1,2,3.
		\end{equation}
		Part \ref{parity} follows from the uniqueness given by the Cauchy--Lipschitz theorem for  the solution of \eqref{profile}--\eqref{profile-kt}
		with initial condition \eqref{profile-initial}, and the invariance
		of \eqref{profile}--\eqref{profile-kt}--\eqref{profile-initial}  under the transformations
		\begin{align*}
		\m(x)&=(m_1(x),m_2(x),m_3(x))\ \rightarrow \  (m_1(-x), -m_2(-x), -m_3(-x)),\qquad
		\\
		\n(x)&=(n_1(x),n_2(x),n_3(x))\ \rightarrow \   (-n_1(-x), n_2(-x), n_3(-x)),\qquad
		\\
		\b(x)&=(b_1(x),b_2(x),b_3(x))\ \rightarrow \  (-b_1(-x), b_2(-x), b_3(-x)),
		\end{align*}
		Setting $W_{j}=\rho_j e^{i \phi_j}$, with $\rho_j\geq 0$ and $\phi_j\in[0,2\pi]$, the asymptotics for $\m_j$ and it derivative $\m'_j$, for $j\in\{1,2,3\}$ in part \ref{asymptotics} are a direct consequence of the asymptotic behavior of  the profile established in Proposition~\ref{mnb-formulae} and 
		Corollary~\ref{cor-asymp} (recall also that $m'(x)=ce^{\alpha\frac{x^2}{4}}n(x)$ and $\gw=\gm+i\gn$).
		In particular the following limits exist:
		$$
		W_{j}:=\lim_{x\rightarrow \infty} e^{ic\Phi_\alpha(x)} (m_j+in_j)(x), \quad B_j:=\lim_{x\to\infty}b_{j}(x). 
		$$
		Notice also that the relations in  \eqref{relation} implies that $\rho_j^2+B_j^2=1$, and the fact that $\b$ is unitary implies that
		$B_{1}^2+B_{2}^2+B_{3}^2=1$, and that $\rho_1^2+\rho_2^2+\rho_3^2=2$.
		
		The convergence in part \ref{convergence} is an immediate consequence of the definition of $\mm$ in terms of the profile $\m$, and the asymptotics established in part \ref{asymptotics}. The relations between $\rho^{-}_{j}$ and $\rho_{j}$, for $j\in\{1,2,3\}$ follow from the  parity relations for the profile $\m$ established in part \ref{parity}.
		
		It remains to prove part \ref{IVP}. For  $\varphi \in W^{1,\infty}(\R)$,  bearing in mind 
		\eqref{w-asymp} and \eqref{est-w}, it suffices to show that
		\begin{equation}
		\label{limit-origen}
		\lim_{t\to 0^+}\int_\R e^{-i c \Phi_{\alpha}(x/\sqrt{t})} \varphi(x)dx=0, \ \text{ and }\  \lim_{t\to 0^+} \int_\R \big(1+ \frac{\abs{x}}{\sqrt t}\big )e^{-\frac{\alpha x^2}{4t}}\abs{\varphi(x)}dx=0.
		\end{equation}
		The second limit is a direct consequence of following the explicit computations:
		\bq
		\label{integrales-exp}
		\int_\R e^{-\frac{\alpha x^2}{4t}}dx=\sqrt{\frac{\pi t}\alpha},
		\quad 
		\text{ and }\quad 
		\int_\R \abs{x}e^{-\frac{\alpha x^2}{4t}}dx=\frac{4 t}\alpha.
		\eq
		For the first limit in \eqref{limit-origen}, we integrate by parts to obtain
		\begin{align*}
		ic \int_\R e^{-i c \Phi_{\alpha}(x/\sqrt{t})}\varphi(x)dx=&
		-\sqrt{t}\int_\R \big( e^{-i c \Phi_{\alpha}(x/\sqrt{t})} )' e^{-\frac{\alpha x^2}{4t}} \varphi(x)dx\\
		=&\int_\R e^{-i c \Phi_{\alpha}(x/\sqrt{t})-\frac{\alpha x^2}{4t}} \left(\sqrt{t}\varphi'(x)-\frac{\alpha x}{2\sqrt{t}}\varphi(x) \right)dx.
		\end{align*}
		Since $\abs{\varphi}$ and $\abs{\varphi'}$ are bounded on $\R$, the conclusion follows again by using \eqref{integrales-exp}.
	\end{proof}
	\medskip

	\section{Limiting behaviour of the trajectories of the profiles}
	\label{sec-limit-circle}
	In this section we prove Theorem~\ref{plane}.  In what follows  we denote  $\mathcal{C}^{\pm}_{c,\alpha}$ the great circle $\mathcal{C}_{c,\alpha}={\mathcal{P}}^{\pm}_{c,\alpha}\cap \mathbb{S}^2$, with $\mathcal{P}^{\pm}_{c,\alpha}$ being the planes passing through the origin with (unitary) normal vectors
	$$
	\bm B^{+}_{c,\alpha}=(B_{1,c,\alpha}, B_{2,c,\alpha}, B_{3,c,\alpha})\qquad {\hbox{and}}\qquad  
	\bm B^{-}_{c,\alpha}=(-B_{1,c,\alpha}, B_{2,c,\alpha}, B_{3,c,\alpha}),
	$$
	given by Theorem~\ref{thm-conv}.
	
	We will show that the trajectories of the profiles $\m_{c,\alpha}$ converge to the great circles $\mathcal{C}^{\pm}_{c,\alpha}$  as $x\to \pm \infty$, and study the behavior of the limit circles $\mathcal{C}^{\pm}_{c,\alpha}$ with respect to the parameters $c$ and $\alpha$, analyzing the angle $\vartheta_{c,\alpha}\in[0,\pi]$ between their normal vectors $\bm B^{+}_{c,\alpha}$ and $\bm B^{-}_{c,\alpha}$.
	
	We start this section stating a corollary of Theorem~\ref{thm-conv} that will be used in what follows.  Precisely, recalling that $\gw=\gm+i\gn$,  from \eqref{est-w} and using the constants defined in Theorem~\ref{thm-conv},
	we have the following
	\begin{cor}
		\label{cor-facil}
		For $ x\geq 1$ and $j\in \{1,2,3\}$, we have
		\begin{align*}
		m_{j,c,\alpha}(x)&= \rho_{j,c,\alpha} \cos(c\Phi_\alpha({x})-\phi_j)+\boR_j(x), 
		\quad n_{j,c,\alpha}(x)=- \rho_{j,c,\alpha} \sin(c\Phi_\alpha({x})-\phi_j)+\tilde \boR_j(x),
		\end{align*}
		for some functions $\boR_j$ and $\tilde \boR_j$ satisfying  the bounds
		$\abs{\boR_j(x)}, \abs{\tilde \boR_j(x)}\leq 10\beta/( c \alpha^2){x}e^{-\alpha x^2/4}$.
	\end{cor}
	We will also use the two lemmas below in the proof of Theorem~\ref{plane}. The first one establishes some relations between the constants appearing  in the asymptotics of the profile $\m_{c,\alpha}$ that can be
	deduced by using some geometric properties of the Serret--Frenet system.
	\begin{lemma}
		\label{lem-constants}
		The constants given by Theorem~\ref{thm-conv} satisfy the following identities
		\begin{equation}
		\label{iden1}
		\begin{gathered}
		B_{1,c,\alpha}=\rho_{2,c,\alpha}\rho_{3,c,\alpha} \sin(\phi_{3,c,\alpha}-\phi_{2,c,\alpha}), \ B_{2,c,\alpha}=\rho_{1,c,\alpha}\rho_{3,c,\alpha}\sin(\phi_{1,c,\alpha}-\phi_{3,c,\alpha}),\\
		B_{3,c,\alpha}=\rho_{1,c,\alpha}\rho_{2,c,\alpha}\sin(\phi_{2,c,\alpha}-\phi_{1,c,\alpha}),
		\end{gathered}
		\end{equation}
		and
		\begin{gather}
		\label{iden2}
		B_{1,c,\alpha} \rho_{1,c,\alpha} e^{ i \phi_{1,c,\alpha}}+B_{2,c,\alpha} \rho_{2,c,\alpha} e^{ i \phi_{2,c,\alpha}}+B_{3,c,\infty} \rho_{3,c,\alpha} e^{ i \phi_{3,c,\alpha}}=0,\\
		\label{iden3}
		\rho_{1,c,\alpha}^2 e^{2 i \phi_{1,c,\alpha}}+\rho_{2,c,\alpha}^2 e^{2 i \phi_{2,c,\alpha}}+\rho_{3,c,\alpha}^2 e^{2 i \phi_{3,c,\alpha}}=0.
		\end{gather}
	\end{lemma}
	\begin{proof}
		Dropping the subscripts $c$ and $\alpha$,  
		using the relation 
		$\b=\m\times \n$ and the asymptotics in Corollary~\ref{cor-facil}, 
		we get for $x\geq1$,
		\begin{align*}
		b_1(x)&=	-\rho_2\rho_3\big(\cos(c\Phi_\alpha(x)-\phi_2)\sin(c\Phi_\alpha(x)-\phi_3)
		- \cos(c\Phi_\alpha(x)-\phi_3)\sin(c\Phi_\alpha(x)-\phi_2)
		\big)+o(1),\\
		b_2(x)&=	-\rho_3\rho_1\big( \cos(c\Phi_\alpha(x)-\phi_3)\sin(c\Phi_\alpha(x)-\phi_1)
		- \cos(c\Phi_\alpha(x)-\phi_1)\sin(c\Phi_\alpha(x)-\phi_3)
		\big)+o(1),\\
		b_3(x)&=	-\rho_1\rho_2\big( \cos(c\Phi_\alpha(x)-\phi_1)\sin(c\Phi_\alpha(x)-\phi_2)
		- \cos(c\Phi_\alpha(x)-\phi_2)\sin(c\Phi_\alpha(x)-\phi_1),
		\big)+o(1),
		\end{align*}
		where $o(1)$ is a function of $x$, which  depends on the parameters $\alpha$ and $c$, 
		that converges to $0$ as $x\to \infty$.
		Noticing that, for $k,j\in\{1,2,3\}$,
		$$-\cos(c\Phi_\alpha(x)-\phi_j)\sin(c\Phi_\alpha(x)-\phi_k)
		+ \cos(c\Phi_\alpha(x)-\phi_k)\sin(c\Phi_\alpha(x)-\phi_j)=\sin(\phi_k-\phi_j),
		$$
		we obtain
		\begin{align*}
		b_1(x)&=\rho_2\rho_3\sin(\phi_3-\phi_2)+o(1),\\
		b_2(x)&=\rho_1\rho_3\sin(\phi_1-\phi_3)+o(1),\\
		b_3(x)&=\rho_1\rho_2\sin(\phi_2-\phi_1)+o(1).
		\end{align*}
		Letting $x\to\infty$, in the above identities we obtain \eqref{iden1}.
		
		To establish the other identities, we first recall that the vectors $\m$, $\n$ and $\b$ satisfy the following relations
		$$
		\m\cdot \n=\n\cdot \b=\b\cdot \m=0\qquad {\hbox{and}}\qquad |\m|=|\n|=|\b|=1.
		$$
		Now, from the asymptotics in Corollary~\ref{cor-facil} and the identity
		$$(\m+i\n)\cdot\b =0,
		$$
		we have
		\begin{equation*}
		b_{1}(x)\rho_1 e^{-i(c\Phi_\alpha(x)-\phi_1)}+b_2(x)\rho_2 e^{-i(c\Phi_\alpha(x)-\phi_2)}+b_{3}(x) \rho_3 e^{-i(c\Phi_\alpha(x)-\phi_3)}=o(1),
		\end{equation*}
		for $i\in\{ 1,2,3\}$. Dividing  by 
		$e^{-ic\Phi_\alpha(x)}$ and letting $x\to\infty$, we obtain formula \eqref{iden2}. Finally, formula \eqref{iden3} follows 
		easily using a similar argument, bearing in mind the orthogonality relation  $ \m\cdot \n=0,$ and that $2\cos(y)\sin(y)=\Im (e^{2iy})$, for $y\in \R$.
	\end{proof}
	
	\begin{remark}
		Although we do not use \eqref{iden3} in this work, this relation could be helpful in establishing
		further properties of the solutions.
	\end{remark}
	
	Next we study the angle between the normal vectors to the great circles $\mathcal{C}^{\pm}_{c,\alpha}$, given by
	$
	\vartheta_{c,\alpha}=\arccos(2B_{1,c,\alpha}^2-1),
	$
	with $\vartheta_{c,\alpha}\in [0,\pi]$. We have the following.
	\begin{lemma}
		\label{lem:ang-c}
		For  $c\geq \beta \sqrt{\pi}/\sqrt{\alpha}$, we have
		\begin{equation}
		\label{est-angulo}
		\vartheta_{c,\alpha}\geq \arccos\left( -1+\frac{2\pi\beta^2 }{c^2\alpha}\right).
		\end{equation}
	\end{lemma}
	\begin{proof}
		Using the formula \eqref{b-asymp} in Proposition~\ref{mnb-formulae} for $b_{1,c,\alpha}$ with $x=0$,
		we get
		\begin{equation*}
		\abs{ b_{1,c,\alpha}(0)-B_{1,c,\alpha}}\leq \frac{\beta}{2c}
		\int_0^\infty \left(1+\frac{\alpha s^2}{2} \right)e^{-\alpha s^2/4 }.
		\end{equation*}
		Noticing that $b_{1,c,\alpha}(0)=0$, and that
		$$\int_0^\infty \left(1+\frac{\alpha s^2}{2} \right)e^{-\alpha s^2/4 }=\frac{2\sqrt{\pi}}{\sqrt{\alpha}},$$
		we conclude that
		\begin{equation*}
		2B^2_{1,c,\alpha} \leq \frac{2\pi\beta^2}{c^2 \alpha}.
		\end{equation*}
		Since $c\geq \beta\sqrt{\pi}/\sqrt{\alpha}$, we have
		$-1+2\pi\beta^2 /(c^2\alpha)\in [-1,1]$, so we can use that the function $\arccos$
		is decreasing to obtain \eqref{est-angulo}.
	\end{proof}
	We continue to prove Theorem~\ref{plane}.
	\begin{proof}[Proof of Theorem~\ref{plane}]
		As usual, we omit the  subscripts $c$ and $\alpha$ when there is no confusion. 
		In view of the symmetries established in Theorem~\ref{thm-conv}, it is enough to prove the theorem
		for $x\to \infty$. We start noticing that $\abs{\bm B_{c,\alpha}}=1$, so that
		the distance between $\m$ and $\boP^+$ is given by
		\begin{equation*}
		\dist(\m(x),\boP^+)=
		\abs{ m_1(x)B_{1}+m_2(x)B_{2}+m_3(x)B_{3}}.
		\end{equation*}
		To compute the leading term, we notice that using \eqref{iden2}, we have
		\begin{equation*}
		\sum_{j=1}^3\rho_j B_j \cos(c\Phi_\alpha(x)-\phi_j)=\Re
		\Bigg(
		e^{ic \Phi_\alpha(x)}
		\Bigg(
		\sum_{j=1}^3\rho_j B_je^{-i\phi_j}
		\Bigg)
		\Bigg)=0.
		\end{equation*}
		Thus the estimates in Corollary~\ref{cor-facil} give us
		\begin{equation}\label{proof-distance}
		\dist(\m(x),\boP^+))\leq 
		\frac{30\beta}{ c \alpha^2 }xe^{-\alpha x^2/4},\qquad {\hbox{for }} x\geq 1,
		\end{equation}
		Let us fix $x\geq 1$ and let $\bm Q(x)$ be the orthogonal projection of $\m(x)$ on the plane $\boP^+$,
		so that  $\dist(\m(x),\boP^+)= \dist(\m(x),\bm Q(x))$. We also set $\bm C(x) \in \boC^+$ such that
		$\dist(\m(x),\boC^+)= \dist(\m(x),\bm C(x))$.
		By the Pythagorean theorem,
		$$\abs{\bm Q(x)}^2=1-\dist(\m(x),\bm Q(x))^2$$
		and $$\dist(\m(x),\bm C(x))^2=\dist(\m(x),\bm Q(x))^2+(1-\abs{\bm Q(x)})^2.$$
		Therefore
		$$\dist(\m(x),\bm C(x))^2=2-2(1-\dist(\m(x),\bm Q(x))^2)^{1/2},$$
		and using the elementary inequality $\sqrt{1-y}\geq 1-y$, for $y\in[0,1]$,
		we get
		\begin{equation*}
		\dist(\m(x),\bm C(x))\leq \sqrt{2} \dist(\m(x),\bm Q(x)).
		\end{equation*}
		Combing this estimate with \eqref{proof-distance}, we conclude that
		$$\dist(\m(x),\boC^+))=\dist(\m(x),\bm C(x)))\leq \sqrt{2} \dist(\m(x),\boP^+))\leq
		\frac{30\sqrt{2}\beta}{c \alpha }xe^{-\alpha x^2/4}.
		$$
		The limits in \eqref{dist2} follow at once using the definition of $\mm_{c,\alpha}$ in \eqref{def-m} (recall that $\vartheta_{c,\alpha}\in[0,\pi]$). 
		
		The statement in $(ii)$ is an immediate consequence of Lemma~\ref{lem:ang-c}. This finishes the proof of Theorem~\ref{plane}.
	\end{proof}
	The limit 
	$
	\lim_{c\rightarrow \infty}\vartheta_{c,\alpha} =\pi 
	$
	in part (ii) of Theorem~\ref{plane} helps us to understand the angle between the great circles as $c\to\infty$. However, the behavior of the limit circles $\mathcal{C}^{\pm}_{c,\alpha}$ for  $c$ small is much more involved. We conclude this section with some reflections on the behavior of the limit circles $\mathcal{C}^{\pm}_{c,\alpha}$ when $c$ is small. 
	
	Let us remark that when $c=0$, the explicit solution to \eqref{serret}--\eqref{IC} is given by
	\begin{equation}
	\label{sol-c-0}
	\begin{aligned}
	\m_{0,\alpha}(x)&=(1,0,0),\\
	\n_{0,\alpha}(x)&=(0,\cos(\beta x^2/4),-\sin(\beta x^2/4)),\\
	\b_{0,\alpha}(x)&=(0,\sin(\beta x^2/4),\cos(\beta x^2/4)).
	\end{aligned}
	\end{equation}
	Thus we see that in this limit case, there is a change in the  behavior of the solution: There is no limit circle
	and the vector $\b_{0,\alpha}$ does not have a limit at infinity.
	On the other hand, we know that 
	$(\m_{c,\alpha}(x),\n_{c,\alpha}(x),\b_{c,\alpha}(x))$ are continuous with respect to the to $c$, $\alpha$ and $x$. Therefore,
	\begin{equation}
	\label{dependence-c-0}
	\lim_{c\to 0}b_{1,c,\alpha}(x)=b_{1,0,\alpha}(x)=0, \quad \text{ for all }x\in \R.
	\end{equation}
	Of course, we cannot conclude from \eqref{dependence-c-0} estimates for $B_{1,c,\alpha}$, as $c$ goes to $0$.
	For this reason, we performed some numerical simulations
	for different values of $c$ small. 
	In Figures~\ref{fig-c-small} and \ref{fig-curvas-c-001}, we show one of these simulations, in the  case $\alpha=0.5$
	and $c=0.01$, where we see that $m_{1,c,\alpha}\approx 1$ and 
	$b_{1,c,\alpha}\approx 0$ on $[0,2]$ in agreement with \eqref{sol-c-0} and the continuous dependence on $c$. On the other hand, for $x\geq 8.5$, we have $b_{1,c,\alpha}(x)\approx -1$.
	
	However, it seems difficult to infer from our simulations the behavior of 
	$\bm B_{c,\alpha}$ as $c$ goes to zero. For instance, 
	we have computed numerically $\bm B_{c,\alpha}$, 
	and it is not clear that this quantity converges for $c$
	small. For instance, we have obtained 
	$B_{1,c,\alpha}=-0.99215$, for $c=10^{-12}$, 
	$B_{1,c,\alpha}= -0.992045$, for $c=10^{-14}$, 
	and 
	$B_{1,c,\alpha}= -0.991965$, for $c=10^{-16}$. 
	
	\begin{figure}[h]
		\begin{subfigure}[b]{0.5\textwidth}
			\centering
			\includegraphics[width=1\textwidth,height=0.15\textheight]{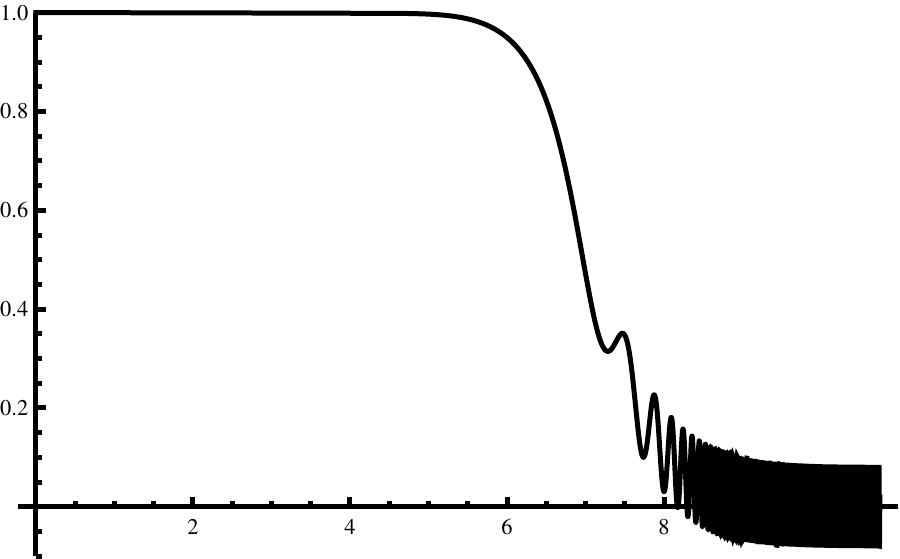}
			\caption{$m_{1,c,\alpha}$}
		\end{subfigure}
		\begin{subfigure}[b]{0.5\textwidth}
			\centering
			\includegraphics[width=\textwidth,height=0.15\textheight]{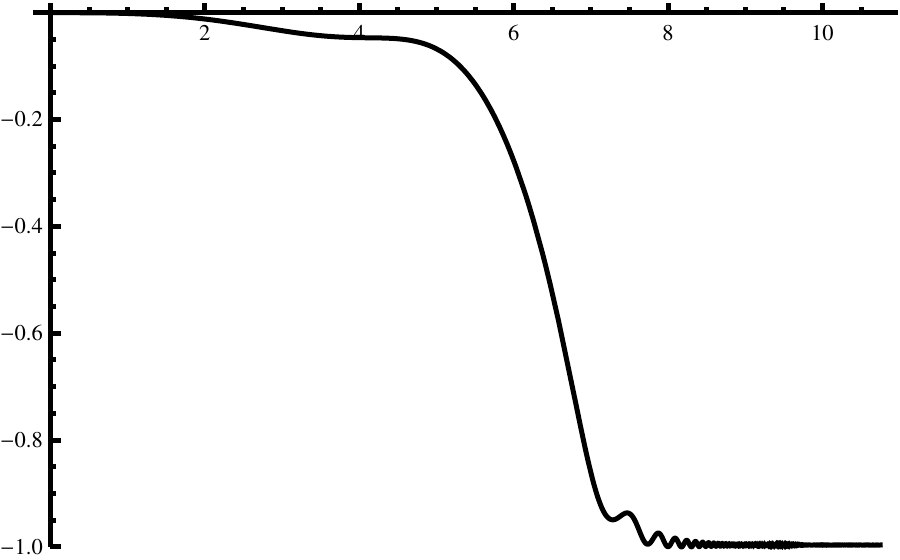}
			\caption{$b_{1,c,\alpha}$}
		\end{subfigure}
		\caption{Functions $m_{1,c,\alpha}$
			and 	$b_{1,c,\alpha}$  for $c=0.01$ and $\alpha=0.5$. 
			The limit at infinity in  (ii) is  $B_{1,c,\alpha}\approx -0.996417$.}
		\label{fig-c-small}
	\end{figure}
	\begin{figure}[h]
		\centering
		\begin{subfigure}[b]{0.5\textwidth}
			\begin{overpic}[trim=0 15mm 0 0,clip, scale=0.9]{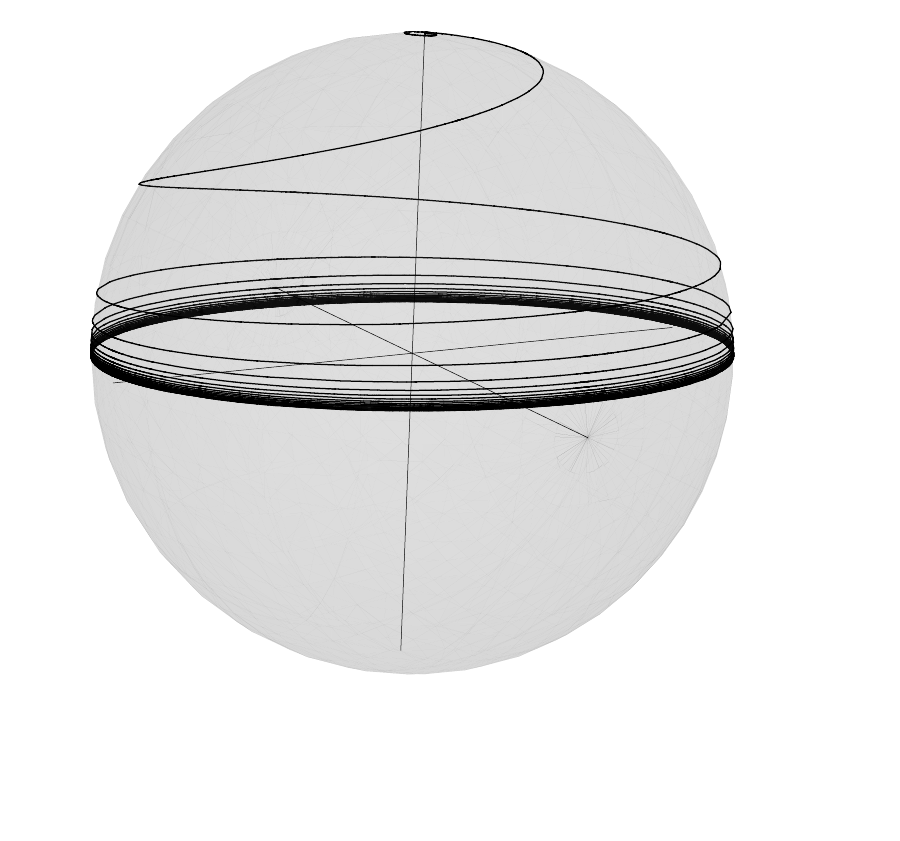}
				\put(46,75){\footnotesize{$m_{1,c,\alpha}$}}
				\put(69,25){\footnotesize{$m_{2,c,\alpha}$}}
				\put(4,32){\footnotesize{$m_{3,c,\alpha}$}}
			\end{overpic}
		\end{subfigure}%
		\caption{Profile $\m_{c,\alpha}$ for $c=0.01$ and $\alpha=0.5$.}
		\label{fig-curvas-c-001}
	\end{figure}
	\bigskip
	
	\noindent {\bf{Acknowledgements. }}  S.~ Guti\'errez was partially supported by  ERCEA Advanced Grant 2014 669689 - HADE. The Universit\' e de Lille also supported S.~Guti\'errez's research visit  during July 2018 through their Invited Research Speaker Scheme.
	A.~de Laire was partially supported by the Labex CEMPI
	(ANR-11-LABX-0007-01), the ANR project ``Dispersive and random waves'' (ANR-18-CE40-0020-01), and the 
	MATH-AmSud project EEQUADD-II.

\bibliographystyle{abbrv}

\end{document}